\documentclass[final,nomarks]{dmtcs-episciences}

\usepackage{amsmath}
\usepackage{amsthm}
\usepackage{amssymb}
\usepackage{nicefrac}
\usepackage{mathtools}
\usepackage[square,comma,numbers,sort&compress]{natbib}
\usepackage{tikz}
\usetikzlibrary{calc,shapes,plotmarks}
\usepackage{footmisc}
\usepackage{enumitem}
\usepackage{microtype}
\usepackage{xurl}
\usepackage{doi}

\definecolor{lightgray}{rgb}{0.8,0.8,0.8}
\definecolor{darkgray}{rgb}{0.7,0.7,0.7}

\hypersetup{
  hidelinks,
  pdftitle={An Assortment of Problems in Permutation Patterns: Unimodality, Equivalence, Derangements, and Sorting},
  pdfauthor={Vincent Vatter},
  pdfsubject={Combinatorics}
}

\theoremstyle{plain}
\newtheorem{theorem}{Theorem}[section]
\newtheorem{proposition}[theorem]{Proposition}
\newtheorem{corollary}[theorem]{Corollary}
\newtheorem{conjecture}[theorem]{Conjecture}
\newtheorem{question}[theorem]{Question}
\newtheorem{problem}[theorem]{Problem}

\theoremstyle{definition}

\newcommand{\Av}{\operatorname{Av}}
\newcommand{\C}{\mathcal{C}}
\newcommand{\D}{\mathcal{D}}
\newcommand{\id}{\mathrm{id}}
\DeclareMathOperator\exc{exc}

\newcommand{\OEISlink}[1]{\href{https://oeis.org/#1}{#1}}

\newcommand{\absdot}[1]{%
	\node[fill,circle,inner sep=0pt,minimum size=4pt] at #1 {};%
}

\newcommand\absdotencircle[1]{%
	\absdot{#1}
	\node[draw,thick,circle,inner sep=0pt,minimum size=8pt] at #1 {};%
}

\newcommand{\plotperm}[1]{%
	\foreach \j [count=\i] in {#1} {%
		\absdot{(\i,\j)}{}
	}
}

\newcommand{\plotpartialperm}[1]{%
	\foreach \i/\j in {#1} {%
		\absdot{(\i,\j)}{}
	}
}

\newcommand{\plotpartialpermencircle}[1]{%
	\foreach \i/\j in {#1} {%
		\absdotencircle{(\i,\j)}{}
	}
}

\newcommand{\plotpermbox}[4]{%
	\draw [thick, line cap=round]
		({#1-0.5},{#2-0.5}) rectangle ({#3+0.5},{#4+0.5});
}

\title[An Assortment of Problems in Permutation Patterns]{An Assortment of Problems in Permutation Patterns: Unimodality, Equivalence, Derangements, and Sorting}
\author{Vincent Vatter}
\affiliation{Department of Mathematics, University of Florida, Gainesville, USA}

\begin{document}
\publicationdata{vol. 28:1, Permutation Patterns 2025}{2026}{10}{10.46298/dmtcs.17543}{2026-02-19; 2026-02-19; 2026-03-02; 2026-11-04}{2026-11-04}
\maketitle

\begin{abstract}
We collect open problems in permutation patterns on four themes: rank-unimodality in the permutation pattern poset and log-convexity of principal classes, Wilf-equivalence and shape-Wilf-equivalence, the enumeration of derangements in permutation classes, and sorting by stacks in series, generalized stacks, and restricted containers ($\mathcal{C}$-machines).
\end{abstract}

\section{Introduction}

This paper accompanies a talk I gave at the pre-conference workshop for early career researchers at \emph{Permutation Patterns 2025}, the 23rd edition of the conference series, held at the University of St Andrews and organized by Christian Bean and Ruth Hoffmann. It is not a list of the best-known or most difficult open problems in permutation patterns, but simply an assortment of problems on a few themes: problems I have encountered, wondered about, or been asked about. Many are of the ``someone really ought to...''\ variety. A few are folklore, and some may be trivial. This is not meant as a roadmap for the field; describing what I feel are the deepest or most important open problems would require a different approach and considerably more buildup. Think of this instead as a collection of interesting attractions that do not take too long to reach.

The months since the initial posting of this article to the arXiv in February 2026 have certainly been eventful ones for mathematics. Large language models have become remarkably capable mathematical tools, and a collection of clearly labeled open problems seems to have offered a tempting place to try them. Consequently, several of the problems collected here have attracted substantial further work, some of it assisted by these tools and not all of it yet available even in preprint form. Keeping pace with all of these developments would significantly enlarge the scope of this project, so I have retained the perspective of the original talk: the problems are presented as they stood in the summer of 2025. Subsequent work is acknowledged briefly in the final section.

The best-known open problem in the field is surely the enumeration of $1324$-avoiding permutations, sequence \OEISlink{A061552} in the OEIS~\cite{sloane:the-on-line-enc:}. This one is easy to reach, but many have looked and progress has been hard to come by. Asked why, Zeilberger replied, ``Because \emph{life is hard}. The few combinatorial objects that we can count exactly are the trivial ones (in the eyes of God)''~\cite{mansour:interview-with-:zeilberger} (emphasis in original). There are, in any case, plenty of places to read more about this conundrum~\cite{marinov:counting-1324-a:,claesson:upper-bounds-fo:,bona:a-new-upper-bou:,johansson:using-functiona:,bona:a-new-record-fo:,egge:defying-god:-th:,conway:on-the-growth-r:,bevan:a-large-set-of-:,conway:1324-avoiding-p:,bevan:a-structural-ch:}. I have nothing new to add about this problem, and so turn to areas that seem more tractable, or at least where we have less evidence of intractability.

The primary meeting place for the permutation patterns community, and for the exchange of open problems, has been the \emph{Permutation Patterns} conference series, founded in 2003 by Michael Albert and Michael Atkinson at the University of Otago in Dunedin, New Zealand. Atkinson, who introduced Albert to the subject~\cite{mansour:interview-with-:}, drew many of the field's current researchers into permutation patterns; his retirement was honored at \emph{Permutation Patterns 2013}. Albert and Atkinson, and the conference series they founded, have shaped the research agenda considerably, and the conference's tradition of open problems sessions has been particularly influential. For his part, Albert became a central figure in permutation patterns, known for his computational and collaborative approach to the subject, before retiring in 2024. In this spirit, many of the problems here beg for computation, either because data is the goal, or just to better illuminate the path forward.

There are earlier collections of open problems that the reader may also find valuable. Wilf, widely regarded as a founder of the field, was the plenary speaker at the inaugural \emph{Permutation Patterns} conference in 2003. His survey ``The patterns of permutations''~\cite{wilf:the-patterns-of:}, written in 1999 but not published until 2002, dates from an early stage in the subject's development; I particularly enjoyed reading it as a graduate student. In that survey, Wilf expressed doubt about the Stanley--Wilf conjecture due to recent results of Alon and Friedgut~\cite{alon:on-the-number-o:}, writing that those results ``seem to make it somewhat less certain because a similar bound, involving the Ackermann function, in the Davenport--Schinzel theory turns out to be best possible.'' Several years later, however, Marcus and Tardos~\cite{marcus:excluded-permut:} proved the conjecture. Wilf also discussed a special case of the Noonan--Zeilberger conjecture~\cite{noonan:the-enumeration:}, reporting that ``Richard Stanley is inclined to favor a negative answer to N--Z's question.'' The general conjecture, which is equivalent to the assertion that every finitely based permutation class has a D-finite generating function, was disproved many years later by Garrabrant and Pak~\cite{garrabrant:permutation-pat:}.%
\footnote{A generating function is \emph{D-finite} if it satisfies a nontrivial linear differential equation with polynomial coefficients. Wilf's formulation concerns permutations with exactly $r$ occurrences of a single pattern~$\tau$, asking whether their generating functions are D-finite for every fixed $r$ and~$\tau$. Noonan and Zeilberger stated the conjecture for avoidance of a finite set of patterns and also stated an apparently more general conjecture allowing prescribed occurrence counts for each pattern in a finite set. Atkinson~\cite[Theorem~2.3]{atkinson:restricted-perm:} gave a short argument which, together with inclusion--exclusion, shows that these two conjectures are equivalent. Garrabrant and Pak's disproof does not settle the single-pattern formulation discussed by Wilf. For simpler counterexamples to the general conjecture, see the final section.}

Another valuable resource is Steingr{\'\i}msson's 2013 survey~\cite{steingrimsson:some-open-probl:}, which covers a broad range of topics, including the M\"{o}bius function of the permutation pattern poset, topological properties of intervals, vincular and mesh patterns, and the structure of growth rates. Steingr{\'\i}msson is himself a leader in the field who has trained many of its current researchers; he delivered the plenary address at \emph{Permutation Patterns~2009}.

For readers seeking broader introductions to permutation patterns, several general references are available. Kitaev's \emph{Patterns in Permutations and Words}~\cite{kitaev:patterns-in-per:} is a comprehensive compendium of the field. B\'ona's undergraduate textbook \emph{A Walk Through Combinatorics}~\cite{bona:a-walk-through-:} contains a very accessible introduction to the area, while his monograph \emph{Combinatorics of Permutations}~\cite{bona:combinatorics-o:} treats the subject in greater depth. Finally, I can't help but recommend my own survey~\cite{vatter:permutation-cla:} in the \emph{Handbook of Enumerative Combinatorics}.

The remainder of the paper is organized as follows. Section~\ref{sec:unimodality} concerns rank-unimodality of intervals in the permutation pattern poset, as well as log-convexity and Stieltjes moment sequences for principal classes. Section~\ref{sec:wilf} discusses symmetries and Wilf-equivalence. Section~\ref{sec:derangements} addresses the enumeration of derangements in permutation classes. Section~\ref{sec:sorting} looks at sorting machines. Before that, the basic definitions are collected below, and the perspective of permutations as relational structures is presented.

\subsection*{Basic definitions}

The basic notions of permutation patterns are easily stated. We always think of permutations as acting on $[n]=\{1,2,\dots,n\}$, and we generally think of permutations in one-line notation, as a list of images $\pi=\pi(1)\pi(2)\cdots\pi(n)$. With apologies to those who view permutations as elements of a Coxeter group (Tenner~\cite{tenner:pattern-avoidan:}), we use the term \emph{length} to mean the number of entries in a permutation, which we denote by $|\pi|$. We denote by $S_n$ the set (group) of all permutations of length~$n$. We identify a permutation~$\pi$ with its \emph{plot}: the set of points $\{(i, \pi(i))\}$ in the plane. When we speak of an entry being to the left or right of, or above or below, another entry, we refer to their relative positions in the plot. Similarly, an entry lies to the \emph{southeast} of another if it is both to the right and below, and so on for the other compass directions.

\begin{figure}
\begin{footnotesize}
\begin{center}
	\begin{tabular}{ccccc}
	\begin{tikzpicture}[scale=0.2, baseline=(current bounding box.center)]
		\plotpermbox{0.5}{0.5}{5.5}{5.5}
		\plotperm{3,2,5,1,4}
	\end{tikzpicture}
	&
	\begin{tikzpicture}[baseline=(current bounding box.center)]
		\node {$\cong$};
	\end{tikzpicture}
	&
	\begin{tikzpicture}[scale=0.2, baseline=(current bounding box.center)]
		\plotpermbox{0.5}{0.5}{9.5}{9.5}
		\plotpartialperm{1/3, 3/2, 4/9, 7/1, 8/8}
	\end{tikzpicture}
	&
	\begin{tikzpicture}[baseline=(current bounding box.center)]
		\node {$\le$};
	\end{tikzpicture}
	&
	\begin{tikzpicture}[scale=0.2, baseline=(current bounding box.center)]
		\plotpermbox{0.5}{0.5}{9.5}{9.5}
		\plotperm{3,6,2,9,5,7,1,8,4}
		\plotpartialpermencircle{1/3, 3/2, 4/9, 7/1, 8/8}
	\end{tikzpicture}
	\end{tabular}
\end{center}
\end{footnotesize}
\caption{The permutation $\sigma = 32514$ (left) is contained in $\pi = 362957184$ (right). The circled entries in~$\pi$ form the subsequence $32918$, which is order-isomorphic to $\sigma$. The permutation~$\pi$ avoids $4321$ because it has no decreasing subsequence of length four.}
\label{fig:plot:example}
\end{figure}
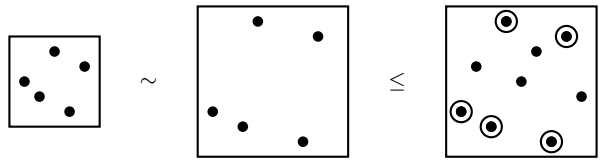

A permutation $\sigma$ is \emph{contained} in~$\pi$, written $\sigma \le \pi$, if~$\pi$ has a subsequence that is \emph{order-isomorphic} to $\sigma$, meaning that its entries appear in the same relative order; otherwise~$\pi$ \emph{avoids} $\sigma$. In terms of plots, $\sigma \le \pi$ means that some selection of points from the plot of~$\pi$ can be transformed into the plot of~$\sigma$ by stretching and shrinking the axes. Figure~\ref{fig:plot:example} illustrates these definitions. This containment relation is a partial order, and we refer to it as the \emph{permutation pattern order}.

A \emph{permutation class} is a downset (or order ideal) in this poset: a set $\mathcal{C}$ of permutations such that if $\pi \in \mathcal{C}$ and $\sigma \le \pi$, then $\sigma \in \mathcal{C}$. We denote by $\C_n$ the set of permutations of length~$n$ in a class $\mathcal{C}$. We sometimes consider the \emph{empty permutation}, which has length zero and therefore lies in every nonempty permutation class. Every permutation class $\mathcal{C}$ is characterized by its \emph{basis}, the set of minimal permutations not in $\mathcal{C}$. We write $\Av(B)$ for the class of permutations avoiding every element of $B$. The basis of a class is necessarily an \emph{antichain} (a set of pairwise incomparable permutations), and since the permutation pattern order contains infinite antichains, bases can be infinite.

\subsection*{Permutations as relational structures}

The study of permutation patterns fits naturally into the broader combinatorics of relational structures. In terms of the plot, a permutation is a set of~$n$ points equipped with two total orders: the left-to-right order (by $x$-coordinate) and the bottom-to-top order (by $y$-coordinate). As Cameron~\cite{cameron:homogeneous-per:} observes, this perspective clarifies what it means for one permutation to be contained in another: containment is simply the induced substructure relation, the same notion studied in the context of graphs (induced subgraphs), posets (induced subposets), integer partitions  (Young's lattice), and other combinatorial objects.

This point of view is relevant in the next section, where we compare unimodality results for permutations with analogous results for partitions and words. In most of those cases, the containment order is the natural one for relational structures of that type, although the answers turn out to be surprisingly different.

\section{Unimodality}
\label{sec:unimodality}

We begin with what is perhaps the easiest of these problems to state, though it may well be among the hardest to prove. A finite sequence $a_0, a_1, \dots, a_n$ is said to be \emph{unimodal} if there exists an index~$k$ such that
\[
    a_0 \le a_1 \le \cdots \le a_k \ge \cdots \ge a_n;
\]
that is, the sequence weakly increases to its maximum and then weakly decreases. A polynomial is said to be \emph{unimodal} if its sequence of coefficients is unimodal. There are scores of unresolved unimodality questions in combinatorics and algebra; for a sample, the reader is referred to Stanley's~survey~\cite{stanley:log-concave-and:}, many of whose conjectures remain open nearly forty years later.

Given a permutation~$\pi$ of length~$n$, let $a_k$ denote the number of permutations of length $k$ contained in~$\pi$. (Each subpermutation is counted once, regardless of how many times it occurs.) Is the sequence $a_0, a_1, \dots, a_n$ necessarily unimodal? To state this in a more refined form, we recall a few definitions. The set of all permutations contained in a given permutation~$\pi$ forms, under the containment order, a \emph{principal downset} (also called a \emph{principal order ideal}). This is a ranked poset, with the rank of a permutation given by its length. A ranked poset is said to be \emph{rank-unimodal} if the sequence counting the number of elements at each rank is unimodal.

\begin{conjecture}
\label{conj-unimodal-downsets}
Every principal downset in the permutation pattern poset is rank-unimodal.
\end{conjecture}

This phenomenon may have first been noted in the permutation pattern context by Murphy, who wrote at the end of his 2003 thesis~\cite[p.~348]{murphy:restricted-perm:}, emphasis added:
\begin{quote}
There exists a program that takes a permutation of arbitrary length (usually about 24) and returns the set of all permutations involved in it, sorted by length. This is a simple but somewhat expensive way of finding the basis of the closure of a given permutation or set of permutations. \emph{The nicest thing about the program is the waisted shape of the output.}
\end{quote}
Murphy then gives the enumeration of permutations contained in a single permutation of length~17, which we present in increasing order of length beginning with the empty permutation:
\[
    1,\ 1,\ 2,\ 5,\ 14,\ 36,\ 87,\ 210,\ 486,\ 927,\ 1315,\ 1348,\ 1005,\ 549,\ 218,\ 61,\ 11,\ 1.
\]

We may go further and ask whether all intervals in this poset are rank-unimodal. The \emph{interval} $[\sigma, \pi]$ is the set of all permutations~$\tau$ satisfying $\sigma \le \tau \le \pi$. It was in this generality that McNamara and Steingr{\'i}msson posed the following conjecture, which subsumes Conjecture~\ref{conj-unimodal-downsets}.

\begin{conjecture}[{McNamara and Steingr{\'i}msson~\cite[Conjecture~9.4]{mcnamara:on-the-topology:}}]
\label{conj-unimodal-intervals}
Every interval in the permutation pattern poset is rank-unimodal.
\end{conjecture}

\subsection*{Analogous structures}

Before surveying what is known for the permutation pattern poset, we briefly consider analogous questions for other combinatorial structures. The picture that emerges is mixed: rank-unimodality holds in some settings but fails in others, which leaves the permutation case genuinely uncertain.

In the \emph{consecutive} pattern poset on permutations, rank-unimodality of intervals was established by Elizalde and McNamara~\cite{elizalde:the-structure-o:}. However, their proof relies on the particularly constrained nature of consecutive containment (every permutation covers at most two others in this order) and does not appear to shed light on the standard (not-necessarily-consecutive) permutation pattern order.

For integer partitions, the situation is murkier. The \emph{containment order} on partitions is defined by~${\mu \le \lambda}$ if the Ferrers diagram of $\mu$ fits inside that of $\lambda$; equivalently, if $\mu_i \le \lambda_i$ for all $i$, padding with zeros as needed. This order makes the set of all partitions into a distributive lattice, called \emph{Young's lattice}. Equivalently, this is the induced-substructure order on finite equivalence relations up to isomorphism, with the parts of a partition recording the sizes of the equivalence classes.

The principal downset generated by the rectangular partition $(m, m, \dots, m)$ with $n$ rows of length~$m$ is denoted by $L(m,n)$, and its rank-generating function is a Gaussian polynomial. Gaussian polynomials have been known to be unimodal since the nineteenth century (Sylvester~\cite{sylvester:proof-of-the-hi:}), although a combinatorial proof was not given until the work of O'Hara~\cite{ohara:unimodality-of-:}, exposited colorfully by Zeilberger~\cite{zeilberger:kathy-oharas-co:}.

One might hope that all principal downsets in Young's lattice are rank-unimodal. Indeed, in 1986 Pouzet and Rosenberg~\cite[p.~367]{pouzet:sperner-propert:} asked something far more general: are principal downsets in the natural ordering on any type of finite relational structures rank-unimodal? As it happened, Stanton~\cite{stanton:unimodality-and:} had found infinite families of counterexamples in Young's lattice by the following year.%
\footnote{Stanton's paper was received on October~29, 1987, although it was not published until 1990. It does not cite Pouzet and Rosenberg; its stated motivation is to study unimodality in Young's lattice.}
The smallest of Stanton's counterexamples is the partition $(8,8,4,4)$, which has rank sequence
\[
    1,\ 1,\ 2,\ 3,\ 5,\ 6,\ 9,\ 11,\ 15,\ 17,\ 21,\ 23,\ 27,\ 28,\ \underbrace{31,\ 30,\ 31}_{\mathclap{\text{not unimodal}}},\ 27,\ 24,\ 18,\ 14,\ 8,\ 5,\ 2,\ 1.
\]
The same example also rules out a graph analogue of Conjecture~\ref{conj-unimodal-downsets}: on graphs, counted up to isomorphism, the principal downset of the disjoint union $K_8\cup K_8\cup K_4\cup K_4$ in the induced-subgraph order has the same non-unimodal rank sequence.

Although the answer to their question is negative, Pouzet and Rosenberg~\cite[Corollary~2.11]{pouzet:sperner-propert:} did establish what might be called the ``first half'' of rank-unimodality: in any principal downset of a finite relational structure, the rank sequence is weakly increasing up to at least the middle rank. Combined with the symmetry of $L(m,n)$, this gives another proof of the unimodality of the Gaussian polynomials. Specialized to permutations, it guarantees that the rank sequence of any principal downset in the permutation pattern poset weakly increases up to at least half its length.

Words over a finite alphabet under the subword order are much better behaved. Chase~\cite{chase:subsequence-num:} proved that principal downsets are not merely rank-unimodal but rank-log-concave,%
\footnote{A simplified proof of Chase's result is given in~\cite{vatter:log-concavity-o:}.}
a stronger property we now define.

\subsection*{Log-concavity}

A sequence $a_0,a_1,\dots,a_n$ of nonnegative terms is said to be \emph{log-concave} if ${a_k^2 \ge a_{k-1} a_{k+1}}$ for all ${1 \le k \le n-1}$, and a polynomial $p(x) = a_0 + a_1 x + \cdots + a_n x^n$ is \emph{log-concave} if its sequence of coefficients is log-concave. For positive sequences, log-concavity implies unimodality because it shows that the ratios $a_k/a_{k-1}$ are weakly decreasing, so once this ratio drops below $1$, it remains below $1$.

We cannot hope for log-concavity in the permutation pattern poset. The rank sequence of every interesting downset begins $1, 1, 2, \dots$, but $1^2 < 1 \cdot 2$ so these sequences are not log-concave. Even ignoring the empty permutation, sequences beginning $1, 2, 5$ or $1, 2, 6$ (both common) fail log-concavity because $2^2 < 1 \cdot 5$.

Nevertheless, log-concavity can play a key role in establishing unimodality, via the following classical result due to Ibragimov~\cite{ibragimov:on-the-composit:} and Keilson and Gerber~\cite{keilson:some-results-fo:}.%
\footnote{It should be remarked that the product of two unimodal polynomials is not necessarily unimodal. For example, $1 + x + 3x^2$ is unimodal but $(1 + x + 3x^2)^2 = 1 + 2x + 7x^2 + 6x^3 + 9x^4$ is not.}

\begin{proposition}
\label{prop:log-concave-uni-convolution}
A polynomial $p(x)$ with positive coefficients is log-concave if and only if $p(x) q(x)$ is unimodal for every polynomial $q(x)$ with nonnegative unimodal coefficients.
\end{proposition}

\subsection*{Layered permutations and compositions}

One might hope that rank-unimodality is at least resolved for layered permutations, a particularly tractable permutation class. We briefly recall their definition.

The \emph{direct sum}, or simply \emph{sum}, of permutations~$\pi$ of length~$m$ and~$\sigma$ of length~$n$ is the permutation $\pi \oplus \sigma$ of length~${m+n}$ defined by
\[
    (\pi \oplus \sigma)(i) \;=\;
    \begin{cases}
        \pi(i) & \text{for } 1 \le i \le m, \\
        \sigma(i - m) + m & \text{for } m + 1 \le i \le m + n.
    \end{cases}
\]
Visually, the plot of $\pi \oplus \sigma$ places the plot of~$\sigma$ above and to the right of the plot of~$\pi$. The operation is associative, so expressions like $\pi \oplus \sigma \oplus \tau$ do not require parentheses.

A permutation is \emph{layered} if it can be expressed as a sum of nonempty decreasing permutations; equivalently, the layered permutations are precisely $\Av(231, 312)$. Every layered permutation has a unique decomposition $\pi = \delta_1 \oplus \delta_2 \oplus \cdots \oplus \delta_k$ where each $\delta_i$ is a decreasing permutation, and so layered permutations are in bijection with compositions: the composition $c=c(1)c(2)\cdots c(k)$ corresponds to the layered permutation with consecutive decreasing layers of lengths $c(1)$, $c(2)$, $\dots$, $c(k)$. For example, the composition $(3, 1, 4)$ corresponds to the layered permutation $321 \oplus 1 \oplus 4321 = 321\,4\,8765$. The \emph{size} of a composition is the sum of its parts, which equals the length of the corresponding layered permutation. This is the rank used in the composition orders below.

Under this bijection, the pattern containment order on layered permutations corresponds to the (generalized) \emph{subword order} on compositions. Concretely, a composition $u = u(1) \cdots u(k)$ with $k$ parts is a subword of a composition $w$ with $\ell$ parts if there exist indices $1 \le i_1 < i_2 < \cdots < i_k \le \ell$ such that $u(j) \le w(i_j)$ for all $1 \le j \le k$. Thus, if $\pi$ and~$\sigma$ are the layered permutations corresponding to $u$ and~$w$, respectively, then $\pi \le \sigma$ in the permutation pattern order if and only if $u \le w$ in the subword order.

Sagan~\cite{sagan:compositions-in:} established unimodality for principal downsets of compositions, which might seem to settle the layered permutation case. However, there is a catch: \emph{the order Sagan considers is not the subword order}.

In Sagan's \emph{componentwise order}, a composition $u = u(1) \cdots u(k)$ is contained in $w = w(1) \cdots w(\ell)$ if $k \le \ell$ and $u(i) \le w(i)$ for all $1 \le i \le k$. This is a natural analogue of Young's lattice, ordering compositions by componentwise comparison of their parts. But it is not the subword order, which corresponds to the permutation pattern order on layered permutations and has been more commonly studied (see, for example, Bergeron, Bousquet-M\'{e}lou, and Dulucq~\cite{bergeron:standard-paths-:} and Sagan and Vatter~\cite{sagan:the-mobius-func:}).

The difference is that in the subword order, embeddings need not align initial entries. For instance, the composition $(2)$ is not contained in $(1,2)$ under the componentwise order, yet the corresponding layered permutations satisfy $21 \le 1\,32$ in the permutation pattern order.

Sagan's main result is the following. The proof is short and instructive.

\begin{theorem}[{Sagan~\cite[Theorem~3.3]{sagan:compositions-in:}}]
\label{thm:sagan}
The principal downset of any composition is rank-unimodal under the componentwise order.
\end{theorem}
\begin{proof}
Let $f_w(x)$ denote the rank-generating polynomial for the principal downset of the composition $w = w(1) \cdots w(\ell)$ in the componentwise order. If $u = u(1) \cdots u(k)$ is contained in~$w$, then either~$u$ is empty, or $u(1) \le w(1)$ and $u(2) \cdots u(k) \le w(2) \cdots w(\ell)$. This yields the recurrence
\[
    f_w(x) \;=\; 1 \,+\, x \bigl(1 + x + x^2 + \cdots + x^{w(1)-1}\bigr)\, f_{w(2) \cdots w(\ell)}(x),
\]
with $f_\varepsilon(x) = 1$ for the empty composition.

We proceed by induction on the number of parts of~$w$. The base case~${f_\varepsilon(x) = 1}$ is trivially unimodal. For the inductive step, the polynomial $1 + x + x^2 + \cdots + x^{w(1)-1}$ is log-concave, so by Proposition~\ref{prop:log-concave-uni-convolution}, its product with the (inductively) unimodal polynomial $f_{w(2) \cdots w(\ell)}(x)$ is unimodal, and multiplying by~$x$ preserves unimodality. Since the coefficient of $x$ in the resulting product equals $1$, adding the constant term $1$ preserves unimodality.
\end{proof}

Although the componentwise and subword orders are different, for compositions with weakly decreasing part sizes, the principal downsets under the two orders have the same rank sequences. To see this, observe that if $w = w(1) \cdots w(\ell)$ is weakly decreasing, then any embedding of $u$ into $w$ in the subword order can be shifted left to produce an embedding in the componentwise order. Thus the set of compositions contained in~$w$ is the same under both orders, even though the order relations among them may differ, and so the principal downsets of compositions with weakly decreasing parts are rank-unimodal by Sagan's Theorem~\ref{thm:sagan}. Reversing compositions is a symmetry of the subword order, so unimodality also holds when~$w$ is weakly increasing.

\begin{corollary}
\label{cor:sagan-monotone}
The principal downset of a composition with monotone parts is rank-unimodal in the subword order. Hence, the principal downset of the corresponding layered permutation is rank-unimodal.
\end{corollary}

In particular, this covers the case of ``rectangular'' compositions $(r, r, \dots, r)$, which was the original question I asked Sagan about at \emph{Permutation Patterns 2007}. The general case, even for layered permutations, remains open.

\begin{conjecture}
\label{conj-layered-unimodal}
The principal downset of every composition is rank-unimodal in the subword order. Hence, the principal downset of every layered permutation is rank-unimodal.
\end{conjecture}

Albert and I have verified this conjecture for every composition of size at most $34$. This is of course a special case of Conjecture~\ref{conj-unimodal-downsets}, which is itself a special case of Conjecture~\ref{conj-unimodal-intervals}.%
\footnote{See the final section for updates.}

\subsection*{Log-convexity and principal classes}

The principal downsets we have been discussing are permutation classes, but rather unusual ones, because they are finite. What can be said along these lines about more typical permutation classes, whose enumerations grow indefinitely?

For \emph{principal classes} (those defined by avoiding a single permutation) there is an intriguing open problem. It is not about unimodality: the enumeration of $\Av(\beta)$ is trivially unimodal for $|\beta|\ge 2$ because it is weakly increasing.%
\footnote{For $|\beta|\ge 2$, if $\pi$ avoids $\beta$ then so does either $\pi \oplus 1$ or $\pi$ with a new maximum prepended, depending on $\beta$.}
Nor is it about log-concavity: sequences beginning~${1, 2, 5}$ or~${1, 2, 6}$ are not log-concave, so $\Av(\beta)$ cannot be log-concave when $|\beta|\ge 3$. The interesting question is whether principal permutation classes are log-convex.

A sequence $a_0, a_1, \dots$ of nonnegative terms is \emph{log-convex} if $a_n^2 \le a_{n-1} a_{n+1}$ for all $n \ge 1$. For example, the factorials are log-convex. The Fibonacci numbers $F_n$ (which enumerate several permutation classes, though no principal ones) provide a nice non-example: Cassini's identity states that $F_{n-1} F_{n+1} - F_n^2$ alternates between $1$ and $-1$, so the log-convexity inequality holds only half the time.

Could principal classes be log-convex? Liang and Sagan~\cite{liang:log-concavity-a:} develop a method using the FKG inequality to establish log-convexity combinatorially; among other results, they prove the log-convexity of the Catalan numbers using the middle order of Bouvel, Ferrari, and Tenner~\cite{bouvel:between-weak-an:}. There is in fact a stronger property that one can ask of principal classes.

That stronger property is the following. A sequence $a_0, a_1, \dots$ is a \emph{Stieltjes moment sequence} if there exists a nonnegative measure $\mu$ supported on $[0, \infty)$ such that $a_n = \int_0^\infty x^n \, d\mu(x)$ for all $n$. Equivalent characterizations~\cite[Theorem~1]{bostan:stieltjes-momen:} are that the Hankel matrix
\[
\begin{bmatrix}
a_0 & a_1 & a_2 & \cdots \\
a_1 & a_2 & a_3 & \cdots \\
a_2 & a_3 & a_4 & \cdots \\
\vdots & \vdots & \vdots & \ddots
\end{bmatrix}
\]
is totally nonnegative (all minors are nonnegative), or that the generating function has a continued fraction expansion
\[
\sum_{n \ge 0} a_n x^n = \cfrac{\alpha_0}{1 - \cfrac{\alpha_1 x}{1 - \cfrac{\alpha_2 x}{1 - \cdots}}}
\]
for some sequence of nonnegative reals $\alpha_0, \alpha_1, \alpha_2, \dots$, where the equality is interpreted as an identity of formal power series.

For example, the Catalan numbers satisfy
\[
C_n = \frac{1}{2\pi} \int_0^4 x^n \sqrt{\frac{4-x}{x}} \, dx
\quad
\text{and}
\quad
\sum_{n \ge 0} C_n x^n = \cfrac{1}{1 - \cfrac{x}{1 - \cfrac{x}{1 - \cdots}}}.
\]

In particular, the Hankel characterization shows immediately that Stieltjes moment sequences are log-convex: among the $2 \times 2$ minors of that matrix are the quantities $a_{n-1} a_{n+1} - a_n^2$.

That the sequence $|\Av_n(12 \cdots k)|$ is a Stieltjes moment sequence follows from work of Rains~\cite{rains:increasing-subs:}, whose interest in this result came from random matrix theory. The first to consider other principal permutation classes seems to have been Elvey Price in his PhD thesis~\cite{elvey-price:selected-proble:}. Using the algebraic generating function first established by B\'{o}na~\cite{bona:exact-enumerati:}, he showed that the enumeration of $\Av(1342)$ is a Stieltjes moment sequence, and posed the following.

\begin{question}[Elvey Price~{\cite[p.~227]{elvey-price:selected-proble:}}]
\label{question:SMS}
Is the enumeration of $\Av(\beta)$ a Stieltjes moment sequence for every permutation $\beta$ of length at least two?
\end{question}

Bostan, Elvey Price, Guttmann, and Maillard~\cite{bostan:stieltjes-momen:} develop this theme further, also stating Question~\ref{question:SMS}. They study the enumerations of $\Av(1342)$ and $\Av(12 \cdots k)$ for $3 \le k \le 8$ in considerable detail, and show that the first fifty terms of the enumeration of $\Av(1324)$ found by Conway, Guttmann, and Zinn-Justin~\cite{conway:1324-avoiding-p:} are consistent with the sequence being a Stieltjes moment sequence. Clisby, Conway, Guttmann, and Inoue~\cite{clisby:classical-lengt:} give numerical evidence that the enumerations of all principal classes defined by avoiding a pattern of length five are Stieltjes moment sequences.

Working more generally with moment sequences of measures on~$\mathbb{R}$, Blitvi\'{c} and Steingr\'{\i}msson~\cite{blitvic:permutations-mo:} have developed a single many-parameter Jacobi continued fraction whose specializations enumerate a wide variety of sets of permutations.

\section{Equivalence}
\label{sec:wilf}

\begin{figure}
\begin{center}
	\begin{footnotesize}
	\begin{tikzpicture}[scale=2.5]
		\path [draw=none] (-2.9,0)--(2.9,0);
		\draw [darkgray, dashed] (0,0)--(1.2,0);
		\draw [darkgray, dashed] (0,0)--(-1.2,0);
		\draw [->] (0,1.4)-- +(0.2,0);
		\draw [->] (0,1.4)-- +(-0.2,0);
		\node at (0,1.4) [above] {$\pi^{\textrm{r}} = \pi\circ\rho$};
		\draw [darkgray, dashed] (0,0)--({1+sqrt(2)/10}, {1+sqrt(2)/10});
		\draw [darkgray, dashed] (0,0)--({-1-sqrt(2)/10}, {-1-sqrt(2)/10});
		\draw [->] ({1+2*sqrt(2)/10}, {1+2*sqrt(2)/10})-- +({sqrt(2)/10}, {-sqrt(2)/10});
		\draw [->] ({1+2*sqrt(2)/10}, {1+2*sqrt(2)/10})-- +({-sqrt(2)/10}, {sqrt(2)/10});
		\node at ({1+1.8*sqrt(2)/10}, {1+1.8*sqrt(2)/10}) [above right] {$\pi^{-1}$};
		\draw [darkgray, dashed] (0,0)--(0,1.2);
		\draw [darkgray, dashed] (0,0)--(0,-1.2);
		\draw [->] (1.4,0)-- +(0,0.2);
		\draw [->] (1.4,0)-- +(0,-0.2);
		\node at (1.4,0) [right] {$\pi^{\textrm{c}} = \rho\circ\pi$};
		\draw [darkgray, dashed] (0,0)--({1+sqrt(2)/10}, {-1-sqrt(2)/10});
		\draw [darkgray, dashed] (0,0)--({-1-sqrt(2)/10}, {1+sqrt(2)/10});
		\draw [->] ({-1-2*sqrt(2)/10}, {1+2*sqrt(2)/10})-- +({sqrt(2)/10}, {sqrt(2)/10});
		\draw [->] ({-1-2*sqrt(2)/10}, {1+2*sqrt(2)/10})-- +({-sqrt(2)/10}, {-sqrt(2)/10});
		\node at ({-1-1.8*sqrt(2)/10}, {1+1.8*sqrt(2)/10}) [above left] {$\left(\pi^{\textrm{rc}}\right)^{-1} = \rho\circ\pi^{-1}\circ\rho$};
		\draw [->] (0.3,0) arc (0:270: 0.3);
		\node at ({0.3*cos(247.5)-0.065}, {0.3*sin(247.5)-0.025}) [below] {$(\pi^{\textrm{r}})^{-1}$};
		\draw [->] (0.5,0) arc (0:180: 0.5);
		\node at ({0.5*cos(157.5)}, {0.5*sin(157.5)}) [left] {$\pi^{\textrm{rc}}$};
		\draw [->] (0.7,0) arc (0:90: 0.7);
		\node at ({0.7*cos(67.5)}, {0.7*sin(67.5)+0.025}) [above] {$(\pi^{\textrm{c}})^{-1}$};
		\draw [darkgray, ultra thick, rounded corners=0.01, line cap=round] (-1,-1) rectangle (1,1);
	\end{tikzpicture}
	\end{footnotesize}
\end{center}
\caption{The symmetries of the square, labelled by their effect on a permutation~$\pi$.}
\label{fig-symmetries-square-perms}
\end{figure}
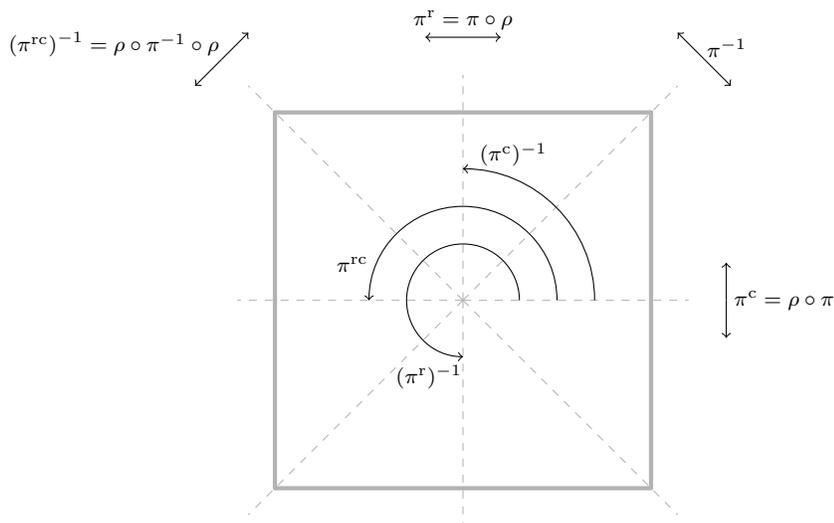

The permutation pattern order has a natural symmetry group, the dihedral group $D_4$ of order eight, visualized in Figure~\ref{fig-symmetries-square-perms} as symmetries of the square. These symmetries act on the plots of permutations, and in the permutation patterns literature the group is commonly described as generated by three reflections: reverse~${\pi \mapsto \pi^{\mathrm{r}}}$ (reflection about the vertical axis), complement~${\pi \mapsto \pi^{\mathrm{c}}}$ (reflection about the horizontal axis), and inverse~${\pi \mapsto \pi^{-1}}$ (reflection about the main diagonal). The standard presentation of dihedral groups instead uses one reflection and one rotation; here, this would be any one of the four reflections together with the $90^\circ$ rotation~${\pi \mapsto (\pi^{\mathrm{c}})^{-1}}$ or the $270^\circ$ rotation~${\pi \mapsto (\pi^{\mathrm{r}})^{-1}}$. As indicated by Figure~\ref{fig-symmetries-square-perms}, composing two reflections yields a rotation, while composing three reflections yields another reflection.

These symmetries partition the permutations of each length into \emph{symmetry classes}: permutations~$\pi$ and $\sigma$ lie in the same symmetry class if $\sigma = \Phi(\pi)$ for some symmetry $\Phi$. Since symmetries are automorphisms of the permutation pattern order, any enumerative question about~$\pi$-avoiding permutations is equivalent to the corresponding question about $\Phi(\pi)$-avoiding permutations. (This applies only when counting all permutations, however; when we impose additional constraints, such as counting derangements in the next section, we may lose symmetries.)

It can be useful to adopt a group-theoretic perspective. For permutations of length~$n$, let
\[
	\rho = {n(n-1)\cdots 21} = \id^{\mathrm{r}} = \id^{\mathrm{c}}.
\]
Then complement is left multiplication by~$\rho$, $\pi^{\mathrm{c}} = \rho \circ \pi$, while reverse is right multiplication by~$\rho$, $\pi^{\mathrm{r}} = \pi \circ \rho$. Since~$\rho$ is an involution, identities among the symmetries become straightforward calculations. For example, $(\pi^{\mathrm{r}})^{-1} = (\pi \circ \rho)^{-1} = \rho^{-1} \circ \pi^{-1} = \rho \circ \pi^{-1} = (\pi^{-1})^{\mathrm{c}}$.

Two permutations can have equinumerous avoidance classes without being symmetries of each other. The classical example is that both $\Av(231)$ and $\Av(321)$ are counted by the Catalan numbers, yet~$231$ and~$321$ lie in different symmetry classes. This motivated Wilf to ask, in the 1980s, for a classification of when avoiding one permutation is equally restrictive as avoiding another.%
\footnote{Babson and West~\cite{babson:the-permutation:} and Stanley~\cite[p.~357]{calkin:herbert-s.-wilf:} both state that Wilf posed this question in the 1980s, though not in a paper at the time. Wilf later included it in~\cite[Section~1]{wilf:the-patterns-of:}. By then, the term ``Wilf-equivalence'' had become widespread, but Wilf did not use it.}
We say that~$\pi$ and~$\sigma$ are \emph{Wilf-equivalent}, written $\pi \sim \sigma$, if $|\Av_n(\pi)| = |\Av_n(\sigma)|$ for all~$n$. Symmetry implies Wilf-equivalence, but as the example $231 \sim 321$ shows, the converse does not hold.

Before discussing Wilf-equivalence further, we pause to note that while counting symmetry classes turns out to be straightforward, this was not always appreciated. In his November 1962 \emph{Scientific American} column, Gardner~\cite{gardner:mathematical-ga:207:5} described the problem of counting ``essentially different'' placements of~$n$ non-attacking rooks on an $n \times n$ chessboard (equivalent to counting symmetry classes of permutations of length~$n$), writing that ``the task of eliminating rotation and reflection duplicates is so difficult that it is not known how many essentially different solutions exist even on as low-order a board as the~${8 \times 8}$.'' In fact, this is a routine application of P\'olya theory, and moreover, Lucas had already solved the problem in his 1891 book \emph{Th\'{e}orie des nombres}~\cite[pp.~220--222]{lucas:theorie-des-nom:}, well before P\'olya.%
\footnote{For a similar problem that is actually difficult, consider the \emph{$n$-queens problem}, which asks to count non-attacking placements of~$n$ queens on an $n \times n$ board. Even without accounting for symmetries, there is no simple formula: the counts for board sizes $1,2,3,\dots$ are $1, 0, 0, 2, 10, 4, 40, 92, 352, \dots$ (\OEISlink{A000170}). The original computation for $n = 27$ required a year-long massively parallel computation~\cite{preuser:putting-queens-:}. We refer to the survey of Bell and Stevens~\cite{bell:a-survey-of-kno:}.}

\begin{table}[t]
\begin{center}
    \begin{tabular}{l l | rrrrrrr}
        & & 1 & 2 & 3 & 4 & 5 & 6 & 7 \\ \hline
        symmetry classes & \OEISlink{A000903}
            & 1 & 1 & 2 & 7 & 23 & 115 & 694 \\
        Wilf-equiv.\ classes & \OEISlink{A099952}
            & 1 & 1 & 1 & 3 & 16 & 91 & 595
    \end{tabular}
\end{center}
\caption{The number of symmetry classes and Wilf-equivalence classes for $1\le n\le 7$.}
\label{table:wilf-equiv}
\end{table}

Wilf-equivalence is harder than this. If $\pi \sim \sigma$, then~$\pi$ and~$\sigma$ must have the same length. The permutations of lengths $1$, $2$, and $3$ each form a single Wilf-equivalence class. For length $4$, there are $7$ symmetry classes but only $3$ Wilf-equivalence classes; establishing that there are exactly $3$ required considerable work, completed by Stankova~\cite{stankova:classification-:} in 1996. In their 2002 paper, Stankova and West~\cite{stankova:a-new-class-of-:} completed the classification up to length~$7$. These values are displayed in Table~\ref{table:wilf-equiv}. The enumeration of Wilf-equivalence classes has remained stuck here since 2002.%
\footnote{See the final section for updates.}

\begin{question}
\label{ques:we:len8}
How many Wilf-equivalence classes of permutations of length~$8$ are there?
\end{question}

A first step toward extending this to length~$8$ would be to compute $|\Av_n(\beta)|$ for each~$\beta$ of length~$8$ to sufficiently large~$n$ to separate the symmetry classes into candidate Wilf-equivalence classes, then to check which apparent equivalences are already explained by known results. Perhaps existing theorems suffice, or perhaps there are new Wilf-equivalences waiting to be discovered.

\subsection*{Sufficient conditions: shape-Wilf-equivalence}

Most known Wilf-equivalences can be explained by a stronger notion of equivalence. A \emph{full rook placement} (or \emph{frp}) of shape $\lambda$ is a nonattacking rook placement on the Ferrers board of shape $\lambda$, justified at the bottom left, with exactly one rook in each row and column. Every permutation~$\pi$ of length $n$ corresponds to the $n \times n$ square frp with rooks in positions $(i, \pi(i))$; this is simply the plot of~$\pi$ with grid lines added. There is a natural containment order on frps: $R$ is contained in $S$ if $R$ can be obtained from $S$ by deleting rows and columns. Restricted to square frps, this coincides with the pattern order on permutations.

We say that an frp \emph{contains} a permutation $\sigma$ if it contains the square frp corresponding to $\sigma$, and otherwise that it \emph{avoids} $\sigma$. Note that the entire square must fit within the Ferrers board; for instance, an frp whose rooks happen to form a $21$-pattern might still avoid $21$ if the top-right corner of the square they determine doesn't lie in the Ferrers board. A general frp can be visualized as a plot enclosed by a staircase boundary (a Dyck path). Permutations $\beta$ and $\gamma$ are \emph{shape-Wilf-equivalent} if, for every Ferrers shape $\lambda$, the number of $\beta$-avoiding frps of shape $\lambda$ equals the number of $\gamma$-avoiding frps of shape $\lambda$. Shape-Wilf-equivalence implies Wilf-equivalence (by restricting to square shapes), but is strictly stronger.

The power of shape-Wilf-equivalence comes from a closure property first observed by Babson and West~\cite{babson:the-permutation:} implicitly in the proofs of their Theorems~1.6 and~1.9 and made explicit by Backelin, West, and Xin~\cite[Proposition~2.3]{backelin:wilf-equivalenc:}.

\begin{proposition}
\label{prop:swe-sums}
If $\alpha$ and~$\beta$ are shape-Wilf-equivalent, then~${\alpha \oplus \gamma}$ and~${\beta \oplus \gamma}$ are shape-Wilf-equivalent for every permutation~$\gamma$.
\end{proposition}

\begin{figure}[t]
\begin{center}
	\begin{tikzpicture}[scale=0.269,baseline=(current bounding box.south)]
	\begin{scope}[shift={(.5,.5)}, darkgray, line cap=round,rounded corners=.5]
		\draw[very thick, fill=lightgray]
			(0,0)--(0,6)--(2,6)--(2,3)--(5,3)--(5,0)--(0,0);
		\foreach \i in {0,1,...,5}
			\draw (\i,0)--(\i,8) (0,\i)--(8,\i);
		\draw (6,0)--(6,8);
		\draw (7,0)--(7,5);
		\draw (0,6)--(7,6);
		\draw (0,7)--(7,7);
		\draw[very thick]
			(0,0)--(0,8)--(7,8)--(7,7)--(7,5)--(8,5)--(8,0)--(0,0);
	\end{scope}
	\plotperm{3,6,8,1,2,5,7,4}
	\end{tikzpicture}
	\begin{tikzpicture}[scale=0.269,baseline=(current bounding box.south)]
		\path [draw=none] (0,0)--(0,2.68);
		\node at (0,2.68) {$\xrightarrow{\phantom{\text{BS}}}$};
	\end{tikzpicture}
	\begin{tikzpicture}[scale=0.269,baseline=(current bounding box.south)]
	\begin{scope}[shift={(.5,.5)}, darkgray, line cap=round,rounded corners=.5]
		\draw[very thick, fill=lightgray]
			(0,0)--(0,6)--(2,6)--(2,3)--(5,3)--(5,0)--(0,0);
		\foreach \i in {0,1,2}
			\draw (\i,0)--(\i,6) (0,\i)--(5,\i);
		\draw (0,3)--(2,3) (0,4)--(2,4) (0,5)--(2,5);
		\draw (3,0)--(3,3) (4,0)--(4,3);
	\end{scope}
	\plotpartialperm{1/3,2/6,4/1,5/2}
	\end{tikzpicture}	
	\begin{tikzpicture}[scale=0.269,baseline=(current bounding box.south)]
		\path [draw=none] (0,0)--(0,2.68);
		\node at (0,2.68) {$\xrightarrow{\phantom{\text{BS}}}$};
	\end{tikzpicture}
	\begin{tikzpicture}[scale=0.269,baseline=(current bounding box.south)]
	\begin{scope}[shift={(.5,.5)}, darkgray, line cap=round,rounded corners=.5]
		\draw[very thick, fill=lightgray]
			(0,0)--(0,4)--(2,4)--(2,3)--(4,3)--(4,0)--(0,0);
		\foreach \i in {0,1,2}
			\draw (\i,0)--(\i,4) (0,\i)--(4,\i);
		\draw (0,3)--(2,3);
		\draw (3,0)--(3,3);
	\end{scope}
	\plotperm{3,4,1,2}
	\end{tikzpicture}	
	\begin{tikzpicture}[scale=0.269,baseline=(current bounding box.south)]
		\path [draw=none] (0,0)--(0,2.68);
		\node at (0,2.68) {$\xrightarrow{\text{BS}}$};
	\end{tikzpicture}
	\begin{tikzpicture}[scale=0.269,baseline=(current bounding box.south)]
	\begin{scope}[shift={(.5,.5)}, darkgray, line cap=round,rounded corners=.5]
		\draw[very thick, fill=lightgray]
			(0,0)--(0,4)--(2,4)--(2,3)--(4,3)--(4,0)--(0,0);
		\foreach \i in {0,1,2}
			\draw (\i,0)--(\i,4) (0,\i)--(4,\i);
		\draw (0,3)--(2,3);
		\draw (3,0)--(3,3);
	\end{scope}
	\plotperm{4,2,3,1}
	\end{tikzpicture}	
	\begin{tikzpicture}[scale=0.269,baseline=(current bounding box.south)]
		\path [draw=none] (0,0)--(0,2.68);
		\node at (0,2.68) {$\xrightarrow{\phantom{\text{BS}}}$};
	\end{tikzpicture}
	\begin{tikzpicture}[scale=0.269,baseline=(current bounding box.south)]
	\begin{scope}[shift={(.5,.5)}, darkgray, line cap=round,rounded corners=.5]
		\draw[very thick, fill=lightgray]
			(0,0)--(0,6)--(2,6)--(2,3)--(5,3)--(5,0)--(0,0);
		\foreach \i in {0,1,2}
			\draw (\i,0)--(\i,6) (0,\i)--(5,\i);
		\draw (0,3)--(2,3) (0,4)--(2,4) (0,5)--(2,5);
		\draw (3,0)--(3,3) (4,0)--(4,3);
	\end{scope}
	\plotpartialperm{1/6,2/2,4/3,5/1}
	\end{tikzpicture}	
	\begin{tikzpicture}[scale=0.269,baseline=(current bounding box.south)]
		\path [draw=none] (0,0)--(0,2.68);
		\node at (0,2.68) {$\xrightarrow{\phantom{\text{BS}}}$};
	\end{tikzpicture}
	\begin{tikzpicture}[scale=0.269,baseline=(current bounding box.south)]
	\begin{scope}[shift={(.5,.5)}, darkgray, line cap=round,rounded corners=.5]
		\draw[very thick, fill=lightgray]
			(0,0)--(0,6)--(2,6)--(2,3)--(5,3)--(5,0)--(0,0);
		\foreach \i in {0,1,...,5}
			\draw (\i,0)--(\i,8) (0,\i)--(8,\i);
		\draw (6,0)--(6,8);
		\draw (7,0)--(7,5);
		\draw (0,6)--(7,6);
		\draw (0,7)--(7,7);
		\draw[very thick]
			(0,0)--(0,8)--(7,8)--(7,7)--(7,5)--(8,5)--(8,0)--(0,0);
	\end{scope}
	\plotperm{6,2,8,3,1,5,7,4}
	\end{tikzpicture}
\end{center}
\caption{The proof of Proposition~\ref{prop:swe-sums} with $\alpha = 231$, $\beta = 312$, and $\gamma = 21$. The shaded region is the shadow cast by copies of $21$. The bijection labeled BS is due to Bloom and Saracino~\cite{bloom:a-simple-biject:}.}
\label{fig:swe-example}
\end{figure}
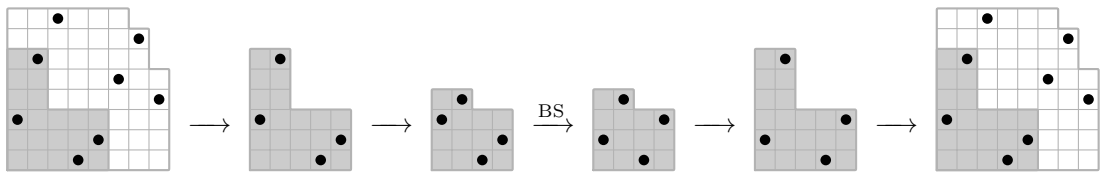

The idea is simple and illustrated in Figure~\ref{fig:swe-example}. For each Ferrers shape~$\lambda$, fix a bijection $f_\lambda$ between the $\alpha$-avoiding and $\beta$-avoiding frps of shape~$\lambda$. Given an $(\alpha \oplus \gamma)$-avoiding frp, consider the \emph{shadow} cast by copies of $\gamma$: the cells lying southwest of some copy of $\gamma$ contained in the frp (where, as always, a copy of $\gamma$ means the full square frp corresponding to $\gamma$, not merely rooks forming a $\gamma$-pattern). This shadow forms a Ferrers board (shown shaded), and the rooks within it, after removing empty rows and columns, form an $\alpha$-avoiding frp of some shape, say~$\mu$. Replace this frp by its $\beta$-avoiding image under~$f_\mu$, which is also an frp of shape~$\mu$. Restoring the empty rows and columns, and leaving the rooks outside the shadow unchanged, yields a $(\beta \oplus \gamma)$-avoiding frp of the original shape.

The key is that rearranging the rooks within the shadow leaves the shadow itself unchanged. For each shadow cell, choose a copy of $\gamma$ northeast of it whose leftmost rook is as far right as possible. Its rooks lie outside the shadow, since a copy northeast of any of them would contradict the choice. These copies remain unchanged. Any cell southwest of a new copy involving a rook inside the shadow is southwest of that rook and so already belongs to the shadow. Thus the shadow is unchanged, and the construction is inverted by applying $f_\mu^{-1}$.

From one shape-Wilf-equivalence, Proposition~\ref{prop:swe-sums} can be used to generate infinitely many others. In fact, only two basic shape-Wilf-equivalences are known for single, classical, permutation patterns.

\begin{theorem}[Backelin, West, and Xin~{\cite[Theorem~2.1]{backelin:wilf-equivalenc:}}]
\label{thm:inc-dec-swe}
For every $k \ge 1$, the permutations $k(k-1) \cdots 21$ and~$12 \cdots k$ are shape-Wilf-equivalent.
\end{theorem}

The case $k = 2$ was proved by West~\cite{west:permutations-wi:} and $k = 3$ by Babson and West~\cite{babson:the-permutation:}. Krattenthaler~\cite{krattenthaler:growth-diagrams:} gave an elegant bijective proof of the general case using Fomin's growth diagrams. De Mier~\cite{mier:k-noncrossing-a:} extended this result to fillings of Ferrers diagrams by nonnegative integers with prescribed row and column sums, and Jel\'\i nek and Mansour~\cite{jeli-nek:wilf-equivalenc:} established analogues for words.

\begin{theorem}[Stankova and West~\cite{stankova:a-new-class-of-:}]
\label{thm:231-312-swe}
The permutations $231$ and $312$ are shape-Wilf-equivalent.
\end{theorem}

A bijective proof was later given by Bloom and Saracino~\cite{bloom:a-simple-biject:}, and this is the bijection labeled~BS in Figure~\ref{fig:swe-example}. This bijection is wonderfully simple: they establish a correspondence between $231$-avoiding frps and labeled Dyck paths, then locally transform the labels, and finally map back to $312$-avoiding frps. Guo, Krattenthaler, and Zhang~\cite{guo:on-shape-wilf-e:} studied analogues for words. Although $231$ and $312$ are not shape-Wilf-equivalent in that setting, they established corresponding equivalences for certain pairs of pattern sets.

Theorem~\ref{thm:inc-dec-swe} states that $12$ and $21$ are shape-Wilf-equivalent, as are $123$ and $321$. Applying Proposition~\ref{prop:swe-sums} to the first of these shows that $123 = 12 \oplus 1$ is shape-Wilf-equivalent to $213 = 21 \oplus 1$.

By Proposition~\ref{prop:swe-sums}, one way to show that two permutations $\alpha$ and~$\beta$ are \emph{not} shape-Wilf-equivalent is to show that $\alpha \oplus 1$ and $\beta \oplus 1$ are not Wilf-equivalent. Enumerating the $(123\oplus 1)$-avoiding, $(231\oplus 1)$-avoiding, and $(132\oplus 1)$-avoiding permutations to length~$7$ rules out any additional equivalences, leaving three shape-Wilf-equivalence classes of permutations of length three: $\{123, 321, 213\}$, $\{231, 312\}$, and $\{132\}$.

Stankova~\cite{stankova:shape-wilf-orde:} showed that these equivalence classes can be ordered by avoidance: for any Ferrers board, there are at least as many $132$-avoiding frps as $321$-avoiding frps, and at least as many $321$-avoiding frps as $231$-avoiding frps. In other words, $132$ is the easiest pattern of length three to avoid on Ferrers boards of any shape, while $231$ is the hardest.

Computation using Ferrers boards contained in an $11\times11$ square shows that Proposition~\ref{prop:swe-sums} together with Theorems~\ref{thm:inc-dec-swe} and~\ref{thm:231-312-swe} account for all shape-Wilf-equivalences of permutations of length $6$ or less.

\begin{question}
Do Proposition~\ref{prop:swe-sums} and Theorems~\ref{thm:inc-dec-swe} and~\ref{thm:231-312-swe} imply all shape-Wilf-equivalences?%
\footnote{See the final section for updates.}
\end{question}

A potential converse to Proposition~\ref{prop:swe-sums} was raised by Burstein at \emph{Permutation Patterns 2025}.

\begin{question}[Burstein; answer: no]
If $\alpha \oplus 1 \sim \beta \oplus 1$, must $\alpha$ and~$\beta$ be shape-Wilf-equivalent?
\end{question}

The answer is no. By direct enumeration on the Ferrers board with row lengths $(6,6,6,6,5,5)$, Stankova and West~\cite[Section~6]{stankova:a-new-class-of-:} showed that the permutations $2413$ and $3142$ are not shape-Wilf-equivalent. However, $2413\oplus1$ and $3142\oplus1$ are inverses and hence Wilf-equivalent.

Although most Wilf-equivalences seem to arise from shape-Wilf-equivalence, there is at least one exception. Stankova~\cite{stankova:forbidden-subse:} proved that $1342 \sim 2413$, but no permutation in the symmetry class of $1342$ is shape-Wilf-equivalent to a permutation in the symmetry class of $2413$, as can be checked by appending a maximum to each pattern and comparing avoidance counts at length~$8$. A bijective proof of this Wilf-equivalence was later given by Bloom~\cite{bloom:a-refinement-of:}.

\subsection*{Toward necessary conditions}

While sufficient conditions for Wilf-equivalence have received considerable attention, necessary conditions have been lacking. The most obvious method for proving that $\pi \not\sim \sigma$ is to enumerate the~$\pi$- and $\sigma$-avoiding permutations until the counts disagree. This is obviously unhelpful for proving general statements about infinite families of patterns, or for developing any structural understanding of when Wilf-equivalence fails.

A more refined approach would study permutations of small ``codimension'' above a pattern. For a nonempty permutation~$\beta$ of length $m$, define
\[
    g_k(\beta) \;=\; \bigl|\{\text{permutations of length } m+k \text{ that contain } \beta\}\bigr|.
\]
Since $|\Av_n(\beta)| = n! - g_{n-m}(\beta)$ for $n\ge m$, we have $\pi \sim \sigma$ if and only if $g_k(\pi) = g_k(\sigma)$ for all $k$. Thus, understanding the function $g_k$ could yield necessary conditions for Wilf-equivalence.

For small $k$, some results are known. Pratt~\cite[p.~276]{pratt:computing-permu:} observed in 1973 that $g_1(\beta) = m^2 + 1$ for any permutation~$\beta$ of length $m$; that is, $g_1$ depends only on the length of~$\beta$, not on its structure. Ray and West~\cite{ray:posets-of-matri:} showed that
\[
    g_2(\beta) \;=\; \frac{m^4 + 2m^3 + m^2 + 4m + 4 - 2j}{2}
\]
for some integer $j$ with $0 \le j \le m-1$. However, the dependence of $j$ on~$\beta$ is not well understood. I raised the following two problems at \emph{Permutation Patterns 2007}~\cite{vatter:problems-and-co:}.

Because $g_1(\beta)$ contains no information about $\beta$ other than its length, $g_2(\beta)$ is the first place that can impose an interesting restriction on Wilf-equivalence: if $\pi \sim \sigma$ then $j(\pi) = j(\sigma)$. This would be a usable necessary condition if we had a structural understanding of $j(\beta)$, but as things stand we do not. Ray and West obtain $j(\beta)$ as the number of ``active lefthand triples'' of a certain pattern, arrived at through a classification of ``insertion synonymities'', and their argument is a finite case analysis that is already substantial at $k=2$.

\begin{problem}
\label{prob:ray:west}
Give a structural interpretation of the quantity $j(\beta)$ in the Ray--West formula for~$g_2(\beta)$, preferably in terms of familiar permutation statistics.%
\footnote{See the final section for updates.}
\end{problem}

Taken literally, a request to express $j(\beta)$ in terms of permutation statistics has a trivial answer, which I do not regard as a solution: the quantity~$j(\beta)$ is already a permutation statistic, and rearranging the Ray--West formula expresses it in terms of $g_2(\beta)$ and the length of~$\beta$. Nor would I regard simply repackaging Ray and West's classification in terms of a newly defined family of permutation patterns as a satisfactory answer. What would be useful is a structural explanation of where $j(\beta)$ comes from, preferably in terms of familiar features of permutations, and ideally leading to a derivation of $g_2(\beta)$ that does not rely on Ray and West's classification. Such a derivation might also extend to the next case, and possibly indicate a way toward a complete characterization of Wilf-equivalence. Of course, it is entirely possible that no such explanation exists.

\begin{problem}
Find a formula for $g_3(\beta)$.%
\footnote{See the final section for updates.}
\end{problem}

\subsection*{Unbalanced Wilf-equivalence}

So far we have discussed Wilf-equivalence of individual permutations, but the notion extends naturally to sets: two bases (antichains of permutations)~$B$ and $B'$ are Wilf-equivalent if $|\Av_n(B)| = |\Av_n(B')|$ for all~$n$. One might expect Wilf-equivalent bases to contain the same number of permutations of each length, but this is not the case.

Atkinson, Murphy, and Ru\v{s}kuc~\cite{atkinson:sorting-with-tw:} characterized the class of permutations sortable by two increasing stacks in series, where the entries in each stack must remain increasing when read from top to bottom. The basis of this class is infinite,
\[
	\{2341,\,254163\} \cup \{2\ (2k{-}1)\ 4\ 1\ 6\ 3\ \cdots\ (2k)\ (2k{-}3) : k \ge 4\},
\]
but they showed that the class is nonetheless Wilf-equivalent to $\Av(1342)$, which was enumerated by B\'ona~\cite{bona:exact-enumerati:}.%
\footnote{A simpler derivation of the generating function of $\Av(1342)$ using full rook placements has since been given by Bloom and Elizalde~\cite[Theorem~4.3]{bloom:pattern-avoidan:}.}
We revisit this Wilf-equivalence in Problem~\ref{prob:bona-bijection}.

Burstein and Pantone~\cite{burstein:two-examples-of:} give further examples of such \emph{unbalanced} Wilf-equivalences between finite sets, such as $\{1324, 3416725\} \sim \{1234\}$. These remain poorly understood.

\section{Derangements}
\label{sec:derangements}

A major focus in the study of permutation patterns has been the enumeration of specific permutation classes, especially those defined by relatively few, relatively short basis elements. Beyond mere ``stamp collecting'', the ability to enumerate a class serves as a proxy for understanding its structure, and as a measuring stick for the adequacy of existing techniques. By now, however, the low-hanging fruit appears to have mostly been picked.

In particular, the ``$2 \times 4$'' classes (those with basis consisting of two permutations of length $4$) have been almost completely exhausted. These classes served for many years as a testbed for different enumerative approaches: they form $56$ symmetry classes and $38$ Wilf-equivalence classes, all but three of which are known to have algebraic generating functions. Albert, Homberger, Pantone, Shar, and Vatter~\cite{albert:generating-perm:} give evidence that the generating functions for these three exceptional classes are not differentially algebraic, which would imply that they are not D-finite.%
\footnote{A generating function $f(x)$ is \emph{differentially algebraic} if it satisfies a nonzero polynomial equation in $x$, $f(x)$, and finitely many derivatives of $f(x)$. See the final section for updates.}
It remains possible that the generating functions for these three exceptional classes have nice continued fraction expressions.

As exact enumeration of permutation classes reaches maturity, it is natural to seek refinements that impose additional structure. One may ask, for instance, to count the alternating (up-down) permutations in a class, or the even permutations, or the Dumont permutations of the first kind, or the involutions. Many such refinements have proved tractable.

However, one type of permutation has remained stubbornly difficult to count, even in very well-behaved permutation classes: derangements. Recall that a \emph{derangement} is a permutation with no fixed points (entries satisfying $\pi(i)=i$, lying on the main diagonal of the plot). In particular, the empty permutation is a derangement, while the permutation $1$ is not. We survey the few known results below, but quickly reach open questions.

\subsection*{Derangements avoiding a pattern of length three}

Robertson, Saracino, and Zeilberger~\cite{robertson:refined-restric:} initiated the study of pattern-avoiding derangements. Both the inverse and reverse-complement symmetries preserve the number of fixed points, so they act as symmetries on the set of derangements. This reduces the consideration of derangements avoiding a pattern of length three to four cases: $\{123\}$, $\{132, 213\}$, $\{231, 312\}$, and $\{321\}$.
%
%

Robertson, Saracino, and Zeilberger proved that for $\beta \in \{132, 213, 321\}$, the~$\beta$-avoiding derangements are counted by Fine's sequence (\OEISlink{A000957}). The counts for lengths up to $12$ are
\[
	1,\ 0,\ 1,\ 2,\ 6,\ 18,\ 57,\ 186,\ 622,\ 2120,\ 7338,\ 25724,\ 91144.
\]
The coincidence between $132$ and $321$ is surprising, since these patterns lie in different symmetry classes. Of course they are Wilf-equivalent (both avoidance classes are counted by the Catalan numbers), but the coincidence for derangements does not follow from that. In fact, Robertson, Saracino, and Zeilberger proved something stronger: the distribution of the number of fixed points is the same across all $132$-avoiding and all $321$-avoiding permutations.

Elizalde~\cite{elizalde:fixed-points-an:} strengthened this further, proving that the joint distribution of the number of fixed points and the number of excedances (entries satisfying $\pi(i) > i$) is the same for $132$-avoiders and $321$-avoiders. Elizalde and Pak~\cite{elizalde:bijections-for-:} later gave a bijective proof of this result.%
\footnote{The chronology here is unusual: Elizalde and Pak's bijective paper appeared in 2004, while Elizalde's initial result, though presented at \emph{FPSAC 2003}, was not published in a journal until 2012, in the Zeilberger Festschrift volume of the \emph{Electronic Journal of Combinatorics}.}

%
%

For $\beta \in \{231, 312\}$, the~$\beta$-avoiding derangements are counted by \OEISlink{A258041}. The counts for lengths up to $12$ are
\[
    1,\ 0,\ 1,\ 1,\ 4,\ 10,\ 31,\ 94,\ 303,\ 986,\ 3284,\ 11099,\ 38024.
\]
Robertson, Saracino, and Zeilberger compute the terms of this sequence through length~$8$~\cite[Table~3]{robertson:refined-restric:}, and prove~\cite[Theorem~7.1]{robertson:refined-restric:} that there are strictly fewer $231$-avoiding derangements than $132$-avoiding derangements of each length~$n\ge 3$ (one might skip ahead to Figure~\ref{fig-derangement-ratios}; it's not even close). Elizalde~\cite[Theorem~3.7]{elizalde:multiple-patter:} gives a continued fraction expression for the generating function for $312$-avoiding permutations according to the number of fixed points.%
\footnote{Given the generality of the continued fraction of Blitvi\'{c} and Steingr\'{\i}msson~\cite{blitvic:permutations-mo:}, one might hope it could enumerate these permutations. Indeed, their continued fraction can be specialized to count derangements and, separately, to count $231$-avoiding permutations. However, the specialization for $231$-avoiding permutations does not keep track of their fixed points, so it does not give an enumeration of $231$-avoiding derangements.}

%
%

This leaves the case of $123$-avoiding derangements. The counts for lengths up to $12$ are
\[
    1,\ 0,\ 1,\ 2,\ 7,\ 20,\ 66,\ 218,\ 725,\ 2538,\ 8646,\ 31118,\ 108430.
\]
This is sequence \OEISlink{A318232} in the OEIS, and that entry references the work of Fu, Tang, Han, and Zeng~\cite{fu:qt-catalan-numb:}. They define
\[
    G_n(t) = \sum_\pi t^{\exc\pi},
\]
where the sum is over $123$-avoiding derangements of length~$n$ and $\exc\pi$ denotes the number of excedances of~$\pi$. In particular, $G_n(1)$ is the number of $123$-avoiding derangements.

Fu, Tang, Han, and Zeng also consider $G_n(-1)$, for which they have a conjecture~\cite[Conjecture~5.11]{fu:qt-catalan-numb:}. For odd~$n$, we always have $G_n(-1) = 0$: the reverse-complement symmetry preserves $123$-avoidance and, for derangements, exchanges excedances with non-excedances. Thus if~$\pi$ has $k$ excedances, its reverse-complement has $n - k$, and when~$n$ is odd these have opposite parities, so the contributions cancel. For even~$n$, they conjecture that $(-1)^{n/2} G_n(-1)$ is positive. This amounts to a statement about whether, for a given even length, there are more $123$-avoiding derangements with an even number of excedances or with an odd number.%
\footnote{See the final section for updates.}

A $123$-avoiding permutation can have at most two fixed points. Elizalde~\cite[Corollary~3.2]{elizalde:multiple-patter:} shows that the enumeration of $123$-avoiding permutations by number of fixed points reduces to enumerating those with exactly two fixed points, as the other cases can be computed from this count. He gives a formula~\cite[Theorem~3.3]{elizalde:multiple-patter:} for the number of $123$-avoiding permutations with exactly two fixed points, but it involves a quintuple sum.

Using his Corollary~3.2, Elizalde~\cite[Theorem~3.4]{elizalde:multiple-patter:} shows that the number of $123$-avoiding derangements of length~$n$ is at least $C_n-2C_{n-1}$ for $n\ge1$, where $C_n$ denotes the $n$th Catalan number. He proves that this lower bound is strictly greater than the $n$th Fine number $F_n$ for $n\ge13$. (Here and below, we repurpose $F_n$ to denote the $n$th Fine number rather than the $n$th Fibonacci number.) Together with direct enumeration for $4\le n\le12$, this establishes that there are strictly more $123$-avoiding derangements than $132$-avoiding derangements of every length~$n\ge4$ (see Figure~\ref{fig-derangement-ratios}).%
\footnote{It would be interesting to prove this inequality by an explicit, length-preserving injection from $132$-avoiding derangements to $123$-avoiding derangements. One might hope to obtain such an injection by restricting a bijection between the full avoidance classes. Claesson and Kitaev~\cite{claesson:classification-:} survey many known bijections between $\Av(321)$ and $\Av(132)$. Taking the inverse of such a bijection and composing with any of the four symmetries taking $\Av(321)$ to $\Av(123)$ gives a bijection from $\Av(132)$ to $\Av(123)$. However, none of the resulting bijections restricts to the desired injection: each sends some $132$-avoiding derangement of length~$4$ to a permutation with a fixed point.}

A simpler expression for this enumeration would be welcome. Elizalde~\cite[p.~8]{elizalde:multiple-patter:} notes that he has ``not been able to find a satisfactory expression'' for the generating function recording length and the numbers of fixed points and excedances in $123$-avoiding permutations. He also expresses hope~\cite[p.~39]{elizalde:multiple-patter:} that the quintuple sum in his Theorem~3.3 can be simplified. Birmajer, Gil, Tirrell, and Weiner~\cite[Conjecture~A.2]{birmajer:pattern-avoidin:} conjecture a formula for the number of $123$-avoiding permutations with two fixed points at a specified distance; if proved, this would reduce Elizalde's quintuple sum to a double sum. The resulting counts agree with those from Elizalde's formula through $n = 200$.%
\footnote{See the final section for updates.}

Even without such a simplification, Elizalde's formulas show that the coefficients are multiple binomial sums, so the generating function is D-finite by a result of Bostan, Lairez, and Salvy~\cite[Corollary~3.6]{bostan:multiple-binomi:}. Using the first $200$ terms, Pantone (personal communication) was able to fit a candidate linear differential equation to the generating function, although the resulting equation is quite unwieldy (far too long to include here). It remains possible that a nicer expression exists in the form of a continued fraction, as Elizalde~\cite[Theorem~3.7]{elizalde:multiple-patter:} found for $312$-avoiding permutations according to the number of fixed points.

\subsection*{Proportions}

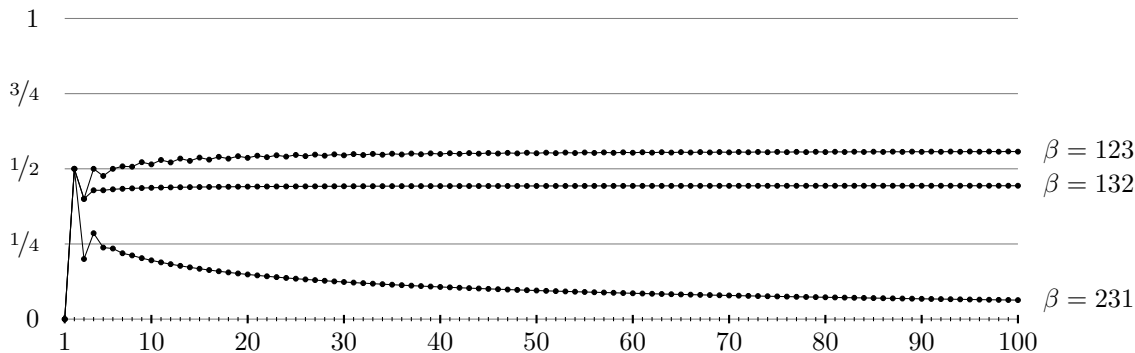
\begin{figure}[t]
\centering
\begin{tikzpicture}[y=5cm, x=.15cm, scale=.8]
	\foreach \y in {0,1,...,4}
		\draw [gray] (1,\y/4)--(100,\y/4);
	\foreach \x in {1,2,...,100}
 		\draw (\x,1pt) -- (\x,-1pt) node[anchor=north] {};
	\foreach \x in {10,20,...,100}
 		\draw [thick] (\x,2pt) -- (\x,-2pt) node[anchor=north] {$\x$};
 	\draw [thick] (1,2pt) -- (1,-2pt) node[anchor=north] {$1$};
	\node[left=6pt] at (1,0) {$0$};
	\node[left=6pt] at (1,.25) {$\nicefrac{1}{4}$};
	\node[left=6pt] at (1,.5) {$\nicefrac{1}{2}$};
	\node[left=6pt] at (1,.75) {$\nicefrac{3}{4}$};
	\node[left=6pt] at (1,1) {$1$};
	\node[right=6pt] at (100, 0.063249932) {$\beta=231$};
	\node[right=6pt] at (100, 0.443702073) {$\beta=132$};
	\node[right=6pt] at (100, 0.557023882) {$\beta=123$};
	\draw plot[mark=*, mark size=.04cm] file {der231-ratios.txt};
	\draw plot[mark=*, mark size=.04cm] file {der132-ratios.txt};
	\draw plot[mark=*, mark size=.04cm] file {der123-ratios.txt};
\end{tikzpicture}
\caption{The proportion of~$\beta$-avoiding permutations that are derangements, by length.}
\label{fig-derangement-ratios}
\end{figure}

Figure~\ref{fig-derangement-ratios} plots the proportion of $\beta$-avoiding permutations of length~$n$ that are derangements, for representatives of the three distinct enumerations arising from patterns of length three. The figure illustrates the comparisons established above: for $n\ge4$, $123$-avoiding permutations are more likely to be derangements than $132$-avoiding permutations, which in turn are more likely to be derangements than $231$-avoiding permutations. It is tempting to wonder whether this agreement with the numbers of fixed points in the forbidden patterns is more than a coincidence.

For~${\beta \in \{132, 213, 321\}}$, we have the Robertson, Saracino, and Zeilberger result that the~$\beta$-avoiding derangements are counted by Fine's sequence. The Fine numbers $F_n$ and Catalan numbers $C_n$ are related by $C_n=2F_n+F_{n-1}$ for $n\ge1$. Thus the proportions $r_n=F_n/C_n$ satisfy
\[
    r_n=\frac12-\frac{n+1}{8n-4}\,r_{n-1},
\]
from which it follows that $r_n\to\nicefrac{4}{9}$.

For $\beta=123$, Figure~\ref{fig-derangement-ratios} suggests that the number of $123$-avoiding derangements grows like $\delta C_n$ for some limiting proportion~$\delta$. Elizalde's lower bound above already implies that the lower limit of the derangement proportion is at least~$\nicefrac{1}{2}$. The limiting proportion~$\delta$ is known from the work of Hoffman, Rizzolo, and Slivken, whose papers relate the shapes~\cite{hoffman:pattern-avoidin:1} and fixed points~\cite{hoffman:pattern-avoidin:2} of random pattern-avoiding permutations to Brownian excursion. As the following two results are not stated explicitly in~\cite{hoffman:pattern-avoidin:2}, we give their short derivations.

\begin{theorem}[Hoffman, Rizzolo, and Slivken, essentially]
\label{thm:123-derangement-proportion}
The limiting proportion of derangements among $123$-avoiding permutations is equal to~$\nicefrac{9}{16}$.
\end{theorem}

\begin{proof}
By~\cite[Lemma~3.4]{hoffman:pattern-avoidin:2}, the number of fixed points in a uniformly chosen $123$-avoiding permutation of length~$n$ converges in distribution to the sum of two independent Bernoulli random variables, each equal to~$1$ with probability~$\nicefrac{1}{4}$. Thus the proportion with no fixed points tends to~$(\nicefrac{3}{4})^2=\nicefrac{9}{16}$.
\end{proof}

The same work also determines the limiting proportion for $231$- and $312$-avoiders.

\begin{theorem}[Hoffman, Rizzolo, and Slivken, essentially]
\label{thm:231-derangement-proportion}
For each $\beta\in\{231,312\}$, the limiting proportion of derangements among $\beta$-avoiding permutations is~$0$.
\end{theorem}

\begin{proof}
Let $X_n$ denote the number of fixed points in a uniformly chosen $231$-avoiding permutation of length~$n$. By~\cite[Theorem~1.1(a)]{hoffman:pattern-avoidin:2}, the random variable $X_n/n^{1/4}$ converges in distribution to an almost surely positive random variable. Consequently, $X_n$ tends to infinity in probability, so the probability that $X_n=0$ tends to~$0$. The result for $312$ follows by inversion, which preserves the number of fixed points.
\end{proof}

These results suggest broader questions. For a permutation class $\mathcal{C}$, let $\mathcal{C}^\circ$ denote the set of derangements in $\mathcal{C}$. The ratio $|\mathcal{C}^\circ_n|/|\mathcal{C}_n|$ need not converge; for example, if $\mathcal{C} = \Av(12)$, the class of decreasing permutations, then $|\mathcal{C}^\circ_n|$, and hence the ratio, oscillates between $0$ and $1$, depending on the parity of~$n$. In addition to this possible oscillation, there is also the trivial issue that for a finite class, the ratio is undefined for sufficiently large~$n$. We therefore restrict attention to infinite classes in asking the following.

\begin{question}
For which infinite permutation classes~$\mathcal{C}$ does the ratio~$|\mathcal{C}^\circ_n|/|\mathcal{C}_n|$ converge?
\end{question}

Numerous follow-up questions suggest themselves. For the class of all permutations, the limit is of course $\nicefrac{1}{e}$. What values in $[0,1]$ are achievable? The case of~$\Av(123)$ gives a limit of~$\nicefrac{9}{16}$, but how much larger can the limit be? Is there a largest possible limit strictly less than $1$?%
\footnote{See the final section for updates.}
For which classes is this limit~$0$?

\subsection*{Separable derangements}

The \emph{skew sum} of permutations~$\sigma$ of length~$m$ and~$\tau$ of length~$n$ is the permutation $\sigma \ominus \tau$ of length~${m+n}$ defined by
\[
    (\sigma \ominus \tau)(i) \;=\;
    \begin{cases}
        \sigma(i) + n & \text{for } 1 \le i \le m, \\
        \tau(i - m) & \text{for } m + 1 \le i \le m + n.
    \end{cases}
\]
Visually, the plot of $\sigma \ominus \tau$ places the plot of~$\tau$ below and to the right of the plot of~$\sigma$.

A permutation is \emph{separable} if it is empty or can be built from the singleton permutation $1$ by repeated application of sum and skew sum. Bose, Buss, and Lubiw~\cite{bose:pattern-matchin:} gave the separable permutations their name and showed that they are precisely $\Av(2413, 3142)$. Notable subclasses include $\Av(132)$ and $\Av(231)$, which we have just discussed, as well as the layered permutations discussed earlier. The nonempty separable permutations are counted by the large Schr\"{o}der numbers (\OEISlink{A006318}); including the empty permutation, the counts for lengths up to $12$ are
\[
	1,\ 1,\ 2,\ 6,\ 22,\ 90,\ 394,\ 1806,\ 8558,\ 41586,\ 206098,\ 1037718,\ 5293446.
\]
This enumeration was first established by Shapiro and Stephens~\cite{shapiro:bootstrap-perco:} in 1991, though they worked with the recursive definition via sums and skew sums. This was before Bose, Buss, and Lubiw had identified the basis $\{2413, 3142\}$, which gives the history an amusing twist: as West~\cite{west:generating-tree:95} reports, Shapiro and Getu later conjectured that the large Schr\"{o}der numbers also count $\Av(2413, 3142)$, the same permutations Shapiro had already counted with Stephens, albeit in a disguised form. West~\cite{west:generating-tree:95} published a proof in 1995, using ``generating trees''.%
\footnote{West wrote that this was ``the first non-trivial enumerative result to be obtained for any problem involving forbidden subsequences of length $k \ge 4$.'' Gessel's 1990 paper~\cite{gessel:symmetric-funct:}, however, gives an explicit formula for $1234$-avoiding permutations. Since West~\cite{west:generating-tree:95} cites Gessel's paper, this claim is puzzling; perhaps he meant the first involving non-monotone patterns, or the first for more than one pattern.}
Stankova~\cite{stankova:forbidden-subse:}, whose paper was submitted after West's but published in 1994, gave a much more natural proof by decomposing the permutations in $\Av(2413,3142)$ into sums and skew sums, showing directly that they are precisely the permutations originally counted by Shapiro and Stephens.

For classes with only finitely many simple permutations (a term we won't define here), such as the separable permutations, Brignall, Huczynska, and Vatter~\cite{brignall:simple-permutat:alg:} show how to systematically obtain algebraic equations for the generating functions of various natural subsets: the alternating permutations, the even permutations, the Dumont permutations of the first kind, the involutions, and many others. However, their method does not apply to derangements, as detailed in~\cite[Section~7]{brignall:simple-permutat:alg:}. The obstacle is that being a derangement does not lie in a finite ``query-complete set of properties'' (in the sense of that paper). The proof of this fact uses displacement sets, $D(\pi) = \{\pi(i) - i : i \in [n]\}$, so a permutation~$\pi$ is a derangement if and only if $0\notin D(\pi)$. While $D(\pi\oplus\sigma)=D(\pi)\cup D(\sigma)$, making it easy to determine if $0\in D(\pi\oplus\sigma)$,
\[
	D(\pi\ominus\sigma) = \left(D(\pi) + |\sigma|\right) \cup \left(D(\sigma) - |\pi|\right),
\]
so nonzero displacements in the components can be shifted to zero, creating fixed points. The need to track these displacements complicates the recursive enumeration.

Thus despite the well-understood structure of separable permutations, no explicit expression for the generating function of separable derangements has been found. These permutations are counted by \OEISlink{A393394}. The counts for lengths up to $12$ are
\[
	1,\ 0,\ 1,\ 2,\ 7,\ 30,\ 124,\ 560,\ 2610,\ 12470,\ 60955,\ 302930,\ 1528621.
\]
These terms were computed by tracking, for each length, how many separable permutations realize each possible displacement set. The count of separable derangements is then the sum over displacement sets not containing $0$.

\begin{problem}
Find an explicit expression for the generating function for the number of separable derangements of length~$n$.%
\footnote{See the final section for updates.}
\end{problem}

The ratios $|\mathcal{S}^\circ_n|/|\mathcal{S}_n|$, where $\mathcal{S}$ denotes the class of separable permutations, appear to be monotonically decreasing for all~$n\ge 5$, approaching something between $\nicefrac{1}{5}$ and $\nicefrac{1}{4}$. However, these guesses are based only on the enumeration up to length~$18$ and so should be taken with a grain of salt.

\begin{question}
\label{ques:sep:derange:ratio}
What is the limiting proportion of derangements among separable permutations, assuming this limit exists?\,%
\footnote{See the final section for updates.}
\end{question}

\section{Sorting}
\label{sec:sorting}

The study of sorting devices has been intertwined with permutation patterns since Knuth's analysis of stack-sorting in \emph{The Art of Computer Programming}. In this section, we consider several sorting machines: multiple stacks in series, Atkinson's enhanced $(r,s)$-stacks, and restricted containers (also known as $\mathcal{C}$-machines). Each machine gives rise to a permutation class (the permutations it can sort/generate), and there are numerous open problems.

\subsection*{Stacks}

A \emph{stack} is a last-in first-out linear sorting device with push and pop operations. The greedy algorithm for stack-sorting a permutation $\pi = \pi(1)\pi(2) \cdots \pi(n)$ proceeds as follows. First, push $\pi(1)$ onto the stack. At a later stage, suppose that the entries $\pi(1), \dots, \pi(i-1)$ have all been either output or pushed onto the stack, so $\pi(i)$ is the next entry in the input. If $\pi(i)$ is less than every entry currently on the stack, push $\pi(i)$ onto the stack. Otherwise, pop entries off the stack (to the output) until $\pi(i)$ is less than every remaining stack entry, then push $\pi(i)$ onto the stack. After all entries have been read, pop any remaining entries from the stack to the output. This produces a permutation~$s(\pi)$.

A permutation is \emph{West-$t$-stack-sortable} if $s^t(\pi)$ is the identity permutation. We caution that for~${t\ge 2}$, the set of West-$t$-stack-sortable permutations does not form a permutation class. As one example,~$35241$ is West-$2$-stack-sortable, because
\[
	s(s(35241))=s(32145)=12345,
\]
but its subpermutation $3241$ is \emph{not} West-$2$-stack-sortable, because
\[
	s(s(3241))=s(2314)=2134\neq 1234.
\]
Indeed, West~\cite[Theorem~4.2.18]{west:permutations-wi:} (and later in~\cite{west:sorting-twice-t:}) showed that the West-$2$-stack-sortable permutations (he did not call them this) are characterized by avoiding $2341$ in the standard sense, and also avoiding the \emph{barred pattern} $3\overline{5}241$; a permutation avoids $3\overline{5}241$ if every occurrence of the pattern~$3241$ can be extended to an occurrence of the pattern~$35241$. 

We will not belabor West stack-sorting any further, and instead turn to sorting with $t$ stacks in series: a permutation is sortable by $t$ stacks in series if there exists \emph{some} sequence of operations that transforms it into the identity, where each operation pushes to or pops from one of the $t$ stacks, and entries pass through the stacks in order (from the first stack to the second, and so on). Unlike the West notion, this definition behaves as one might reasonably hope. In particular, the set of permutations sortable by $t$ stacks in series forms a permutation class: if a permutation can be sorted, then any subpermutation can be sorted by running the same operations while ignoring the entries not present in the subpermutation.

\subsection*{One and two stacks}

The permutations sortable by a single stack are precisely $\Av(231)$. Knuth~\cite[Exercise~2.2.1-5]{knuth:the-art-of-comp:1} stated the equivalent result that a single stack generates precisely $\Av(312)$; we discuss the connection between sorting and generating shortly. For two stacks in series, Tarjan~\cite[Lemma~10]{tarjan:sorting-using-n:} found that the shortest unsortable permutation has length $7$. We quote his ``proof'':%
\footnote{Pratt~\cite{pratt:computing-permu:}, in a paper from the same era, begins a proof with ``We leave to the reader the pleasure of convincing himself that none of the permutations in Figure~3 can be computed by a deque.'' The implicit assumption in both papers, one published in the \emph{Journal of the ACM} and the other in \emph{STOC}, that the reader would verify such claims by hand rather than by computer is now somewhat quaint.}
\begin{quote}
$(2435761)$ is unsortable using two stacks, as the reader may easily verify. Conversely, every sequence of length~$6$ or less may be sorted using two stacks. Exhaustive case analysis will verify this fact.
\end{quote}

In a 1992 technical report~\cite[Theorem~2]{atkinson:sorting-permuta:}, Atkinson seems to have been the first to find all~$22$ basis elements of length~$7$ for two stacks in series. Murphy~\cite[Proposition~257]{murphy:restricted-perm:} proved that the class of permutations sortable by two stacks in series has an infinite basis, and in addition to the $22$ basis elements of length~$7$, he lists the $51$ basis elements of length~$8$. More generally, the number of basis elements of length~$n$, for $n\ge 7$, is given by \OEISlink{A111576}:
\[
	22,\ 51,\ 146,\ 604,\ \dots.
\]
Computing these basis elements up to $n=10$ is not difficult once one views the sortable permutations in terms of products of $231$-avoiding permutations, as we explain below.

\subsection*{Sorting, generating, and duality}

\newcommand{\sortingmachine}[4]{%

	\begin{footnotesize}
	\begin{tikzpicture}[scale=0.3, baseline=(current bounding box.center)]

		\draw [darkgray, very thick, line cap=round] (-1.5,-1.5) rectangle (1.5,1.5);
		\node at (0,0) {\normalsize \ensuremath{#2}};

		\ifthenelse{\equal{#4}{right}}%
		{	

			\node at (-3,.325) {$\xrightarrow{\phantom{\text{BS}}}$};
			\node at ( 3,.325) {$\xrightarrow{\phantom{\text{BS}}}$};

			\begin{scope}[shift={(-8.5,0)}]
				\draw [darkgray, very thick, line cap=round] (-4,-.75)--(4,-.75);
				\node at (0,0) {\ensuremath{#1}};
				\node at (0,-1.5) {\scriptsize input};
			\end{scope}

			\begin{scope}[shift={(8.5,0)}]
				\draw [darkgray, very thick, line cap=round] (-4,-.75)--(4,-.75);
				\node at (0,0) {\ensuremath{#3}};
				\node at (0,-1.5) {\scriptsize output};
			\end{scope}

		}%
		{	

			\node at (-3,.325) {$\xleftarrow{\phantom{\text{BS}}}$};
			\node at ( 3,.325) {$\xleftarrow{\phantom{\text{BS}}}$};

			\begin{scope}[shift={(-8.5,0)}]
				\draw [darkgray, very thick, line cap=round] (-4,-.75)--(4,-.75);
				\node at (0,0) {\ensuremath{#3}};
				\node at (0,-1.5) {\scriptsize output};
			\end{scope}

			\begin{scope}[shift={(8.5,0)}]
				\draw [darkgray, very thick, line cap=round] (-4,-.75)--(4,-.75);
				\node at (0,0) {\ensuremath{#1}};
				\node at (0,-1.5) {\scriptsize input};
			\end{scope}
		}

	\end{tikzpicture}
	\end{footnotesize}
}

We pause to discuss a general framework that clarifies several arguments to follow. We think of sorting machines as transforming an input permutation (generally drawn on the right in diagrams, so that the entries enter the machine in their natural order $\pi(1)$, $\pi(2)$, $\dots$) to an output permutation. Sorting is then the special case where the output is the identity.
\begin{center}
	\sortingmachine{\pi(1)\pi(2)\cdots \pi(n)}{M}{12\cdots n}{left}
\end{center}
Generating is the special case where the input is the identity.
\begin{center}
	\sortingmachine{12\cdots n}{M}{\pi(1)\pi(2)\cdots \pi(n)}{left}
\end{center}

A machine is \emph{symbol oblivious} if its allowed operations are indifferent to the values of the symbols being manipulated. (This terminology follows Atkinson~\cite{atkinson:permuting-machi:}, although he simply uses \emph{oblivious}; we add ``symbol'' to avoid confusion with other uses of the term ``oblivious'' in computer science.) A stack is symbol oblivious: a push or pop is permitted regardless of what symbol is involved.\footnote{Although not under West's notion, since the rule ``pop if the top of the stack is smaller than the next input entry'' depends on comparing symbol values.} Compositions of stacks in series are also symbol oblivious.

For a symbol-oblivious machine, if we relabel the input symbols and perform the same sequence of operations, we obtain the correspondingly relabeled output. Suppose a symbol-oblivious machine~$M$ can transform~$\pi$ into~$\sigma$. Relabeling by an arbitrary permutation~$\tau$ shows that the same sequence of operations allows $M$ to transform $\tau \circ \pi$ into $\tau \circ \sigma$. In particular, taking $\tau = \sigma^{-1}$, we see that $M$ can transform $\sigma^{-1} \circ \pi$ into the identity. In the opposite direction, if $M$ can sort $\sigma^{-1}\circ\pi$, then relabeling by $\sigma$ shows that $M$ can transform $\pi$ into $\sigma$. In other words:

\begin{proposition}
\label{prop:sort-generate}
Let $M$ be a symbol-oblivious machine. Then $M$ can transform $\pi$ into $\sigma$ if and only if $M$ can sort $\sigma^{-1} \circ \pi$. In particular, $M$ can generate $\pi$ if and only if $M$ can sort $\pi^{-1}$.
\end{proposition}

For example, a single stack sorts precisely the class $\Av(231)$. Since $312^{-1} = 231$, Proposition~\ref{prop:sort-generate} implies that a single stack generates the class $\Av(312)$ from the identity.

More generally, symbol obliviousness clarifies what two stacks in series can generate. First, a bit of notation: for classes~$\C$ and $\D$, we define their \emph{composition} by
\[
	(\C\circ\D)_n = \{\pi\circ\sigma : \pi\in\C_n, \sigma\in\D_n\}.
\]
Now suppose entries pass through the first stack and enter the second stack in the order $\sigma(1)$, $\sigma(2)$, $\ldots$, $\sigma(n)$. Since a single stack generates $\Av(312)$, we have $\sigma \in \Av(312)$. By symbol obliviousness, the second stack performs some sequence of operations that would transform $12\cdots n$ into some $\tau \in \Av(312)$. Since the actual symbols input into the second stack are $\sigma(1), \ldots, \sigma(n)$ rather than $1, \ldots, n$, the output is $\sigma \circ \tau$ rather than $\tau$.

Conversely, placing a queue between the two stacks does not change which permutations they can generate, because we can always remove the queue by delaying the operations of the first stack so that it pops each entry only when the second stack is ready to push it. With such a queue in place, given $\sigma, \tau \in \Av(312)$, we let the first stack generate $\sigma$ into the queue, and then let the second stack perform, on the input $\sigma(1), \ldots, \sigma(n)$, the operations that would generate $\tau$ from $12\cdots n$; the output is $\sigma \circ \tau$. Thus, two stacks in series generate precisely $\Av(312) \circ \Av(312)$.

This yields a simple method for computing the basis mentioned earlier. By Proposition~\ref{prop:sort-generate}, two stacks in series sort precisely the inverses of these permutations, which form $\Av(231) \circ \Av(231)$; we generate all elements of this class up to the desired length and identify the minimal permutations that do not appear.

A second useful observation concerns running machines backwards. Suppose that a machine $M$ can transform~$\pi$ into $\sigma$ via some sequence of operations.
\begin{center}
	\sortingmachine{\pi(1)\pi(2)\cdots \pi(n)}{M}{\sigma(1)\sigma(2)\cdots\sigma(n)}{left}
\end{center}
Reading that sequence backwards describes a way for a ``reversed machine'' $M^{\mathrm{r}}$ to transform~$\sigma^{\mathrm{r}}$ into~$\pi^{\mathrm{r}}$. 
\begin{center}
	\sortingmachine{\sigma(1)\sigma(2)\cdots\sigma(n)}{M^{\mathrm r}}{\pi(1)\pi(2)\cdots \pi(n)}{right}
\end{center}
In the diagram above, entries closest to the machine are the first to enter it, while entries farthest from the machine are the first to have left it. In ``temporal order'', the input is $\sigma^{\mathrm{r}}$, and the output is $\pi^{\mathrm{r}}$. This is shown in a more standard form below.
\begin{center}
	\sortingmachine{\sigma(n)\cdots\sigma(2)\sigma(1)}{M^{\mathrm r}}{\pi(n)\cdots\pi(2)\pi(1)}{left}
\end{center}
In general, $M^{\mathrm{r}}$ may be a different machine than $M$, and we say that $M$ is \emph{reversible} if $M^{\mathrm{r}} = M$. Stacks in series are reversible (reverse the roles of input and output, reverse the order of the stacks, and exchange pushes with pops).

If we assume that $M$, and hence also $M^{\mathrm{r}}$, is symbol oblivious, and that $M$ can sort~$\pi$, then running the sorting operations backwards shows that $M^{\mathrm{r}}$ can transform $\rho = n(n-1)\cdots 21$ into $\pi^{\mathrm{r}}$.
\begin{center}
	\sortingmachine{n\ \ \cdots\ \ 2\ \ 1}{M^{\mathrm r}}{\pi(n)\cdots\pi(2)\pi(1)}{left}
\end{center}
Complementing every label makes the input $12\cdots n$ and the output $\pi^{\mathrm{rc}}$, so $M^{\mathrm{r}}$ can generate $\pi^{\mathrm{rc}}$. By Proposition~\ref{prop:sort-generate}, it can therefore sort $(\pi^{\mathrm{rc}})^{-1}$. Since inverse-reverse-complement is an involution, the converse holds as well. We record this symmetry below.

\begin{proposition}
\label{prop:sorting-dual}
If the machine $M$ is symbol oblivious, then $M$ can sort $\pi$ if and only if $M^{\mathrm{r}}$ can sort $(\pi^{\mathrm{rc}})^{-1} = \rho \circ \pi^{-1} \circ \rho$.
\end{proposition}

Since stacks in series are symbol oblivious and reversible, if~$\pi$ can be sorted by $t$ stacks in series, then $(\pi^{\mathrm{rc}})^{-1}$ can also be sorted by $t$ stacks in series. As Figure~\ref{fig-symmetries-square-perms} shows, the permutation $(\pi^{\mathrm{rc}})^{-1}$ is the reflection of~$\pi$ about the anti-diagonal. This permutation has been called the ``dual'' of~$\pi$ (Murphy~\cite[Section~8.1.2]{murphy:restricted-perm:}) or the ``two-stack dual'' (Smith and Vatter~\cite{smith:a-stack-and-a-p:}). We suggest instead the term \emph{sorting dual}, since this duality applies to any symbol-oblivious reversible machine.

\subsection*{General bounds}

For general $t$, Knuth presents an argument in \emph{The Art of Computer Programming}, Volume~3~\cite[Solution to Exercise 5.2.4-19]{knuth:the-art-of-comp:3} showing that if all permutations of length~$n$ can be sorted by $t$ stacks in series, then all permutations of length $2n$ can be sorted by $t+1$ stacks in series. (This argument also appears in Tarjan~\cite[Lemma~11]{tarjan:sorting-using-n:}.)

We sketch the proof. Let~$\pi$ be a permutation of length $2n$ and suppose that all permutations of length~$n$ can be sorted by $t$ stacks in series. By symbol obliviousness, all permutations of length~$n$ can be transformed into $\rho = n(n-1)\cdots 21$. Thus we can use the first $t$ of our $t+1$ stacks to output $\pi(1), \dots, \pi(n)$ in descending order, pushing these entries into the last stack so that they sit in increasing order (read top to bottom) before we read $\pi(n+1)$. We then use the first $t$ stacks to sort $\pi(n+1), \dots, \pi(2n)$, merging these entries with the contents of the final stack into the output as desired.

Murphy~\cite[Proposition~264]{murphy:restricted-perm:} gives a different argument for the same bound. We read the input $\pi(1), \dots, \pi(2n)$ from left to right, pushing each entry onto the first stack. Entries with values greater than~$n$ remain there, while those with values at most~$n$ are immediately popped off and passed to the last $t$ stacks, which sort and output them. Once the entire input has been read and $1, 2, \dots, n$ have been output, the first stack contains precisely $n+1, n+2, \dots, 2n$. Read from top to bottom, these entries occur in the reverse of their order in~$\pi$. We then sort and output them using the last $t$ stacks.

In fact, Knuth's argument and Murphy's argument are the sorting duals of each other. Suppose that the permutations~$\pi$ and~$\sigma$ can be sorted by~$t$ stacks in series. Knuth's argument shows that~${t+1}$ stacks in series can sort all \emph{horizontal juxtapositions} of the form $\begin{bmatrix}\rho\circ\pi&\sigma\end{bmatrix}$: permutations of length~${|\pi|+|\sigma|}$ in which the first $|\pi|$ entries are order-isomorphic to $\rho\circ\pi$ (the complement of~$\pi$) and the last $|\sigma|$ entries are order-isomorphic to~$\sigma$. Murphy's argument shows that $t+1$ stacks in series can sort all \emph{vertical juxtapositions} of the form $\begin{bmatrix} \pi\circ\rho&\sigma\end{bmatrix}^{\mathrm{T}}$: permutations of length $|\pi|+|\sigma|$ in which the largest $|\pi|$ values are order-isomorphic to $\pi\circ\rho$ (the reverse of~$\pi$) and the smallest $|\sigma|$ values are order-isomorphic to~$\sigma$. (In both these statements and in what follows, $\rho$ is to be read as the reverse identity permutation of appropriate length.)

To see the duality, substitute the sorting duals of~$\pi$ and~$\sigma$ into Murphy's construction. Since the sorting dual of~$\pi$ is $\rho \circ \pi^{-1} \circ \rho$, we obtain vertical juxtapositions of the form
\[
	\begin{bmatrix}
	(\rho \circ \pi^{-1} \circ \rho) \circ \rho
	\\[4pt]
	\rho \circ \sigma^{-1} \circ \rho
	\end{bmatrix}
	=
	\begin{bmatrix}
	\rho \circ \pi^{-1}
	\\[4pt]
	\rho \circ \sigma^{-1} \circ \rho
	\end{bmatrix}.
\]
The sorting duals of \emph{these} permutations are precisely the horizontal juxtapositions from Knuth's argument.

Since two stacks in series can sort every permutation of length $6$, it follows from these constructions that $t \ge 2$ stacks in series can sort every permutation of length $3 \cdot 2^{t-1}$. Atkinson~\cite[Corollary, p.~10]{atkinson:sorting-permuta:} refined this bound to $7 \cdot 2^{t-2} - 1$, but the asymptotics are the same: roughly $\log_2 n$ stacks in series suffice to sort all permutations of length~$n$.

For a lower bound, observe that a single stack can sort precisely $C_n$ permutations of length~$n$, where $C_n$ denotes the~$n$th Catalan number. Thus $t$ stacks in series can sort at most $C_n^t$ permutations of length~$n$. Sorting all $n!$ permutations requires $n!\le C_n^t\le4^{nt}$, so Stirling's formula gives a lower bound of roughly $\log_4 n$ stacks.

These elementary bounds leave a factor-of-two gap. Albert, Atkinson, and Linton~\cite[Table~3]{albert:permutations-ge:} give an upper bound of $13.374$ on the growth rate of permutations sortable by two stacks in series. Grouping the stacks in pairs and comparing the resulting count with $n!$ gives a lower bound of roughly $0.534\log_2 n$ (since $2/\log_2(13.374)>0.534$) for the number of stacks in series needed to sort every permutation of length~$n$, narrowing this gap slightly. Thus the order of growth is known to be logarithmic, but there remains a substantial gap between the leading constants in the upper and lower bounds. Knuth posed the following problem more than fifty years ago and rated it~$47$ out of~$50$ in difficulty.

\begin{problem}[{Knuth~\cite[Exercise 5.2.4-20]{knuth:the-art-of-comp:3}}]
Determine the true rate of growth, as $n \to \infty$, of the number of stacks in series needed to sort every permutation of length~$n$.
\end{problem}

\subsection*{Three stacks}

Specializing to three stacks, we immediately run into problems. The exhaustive computational approach that works for two stacks becomes intractable, as computing
\[
	\Av(231) \circ \Av(231) \circ \Av(231)
\]
for even modest lengths is infeasible. We do not even know the length of the shortest permutation that cannot be sorted by three stacks in series, although Atkinson's bound shows that all permutations of length $7\cdot 2-1=13$ can be sorted, so this shortest basis element must have length at least~$14$. Elder and Vatter~\cite{elder:problems-and-co:} report that Elder and Waton wagered a beer on the problem, with Elder guessing $15$ and Waton guessing $22$ (``conjecturing,'' if one is feeling generous).

\begin{question}[{Waton; see Elder and Vatter~\cite{elder:problems-and-co:}}]
\label{ques:shortest:3ss}
What is the length of the shortest permutation that cannot be sorted by three stacks in series?\,%
\footnote{See the final section for updates.}
\end{question}

Murphy~\cite[Conjecture~265]{murphy:restricted-perm:} guesses more generally that $t$ stacks in series can sort all permutations of length up to $(t+1)!$, which would make $25$ the shortest possible length of a permutation unsortable by three stacks in series. However, Murphy's conjecture can't be true in general, since $\log_4 (t+1)!$ grows faster than~$t$.

The known unsortable permutations are shockingly long, given these conjectured answers. Tarjan~\cite{tarjan:sorting-using-n:} reports that he constructed a permutation of length $41$ that cannot be sorted by three stacks in series, although he doesn't present it. Murphy~\cite[Proposition~262]{murphy:restricted-perm:} gives a general method to construct permutations that cannot be sorted by $t+1$ stacks in series, starting from a permutation that cannot be sorted by $t$ stacks. We sketch his argument.

Suppose that some permutation~$\beta$ cannot be sorted by $t$ stacks in series. Consider the permutation
\[
	\pi \;=\; \beta[\,\underbrace{\,\beta^{\mathrm{r}},\, \beta^{\mathrm{r}}, \dots,\, \beta^{\mathrm{r}}}_{\text{$|\beta|-1$ copies}}, 1\,],
\]
the inflation of~$\beta$ formed, loosely speaking, by replacing each entry of~$\beta$ except the last with a block of entries order-isomorphic to $\beta^{\mathrm{r}}$. (Taking~${t=1}$ and $\beta=231$ recovers Tarjan's permutation~$2435761$, which is unsortable by two stacks in series.) We claim that~$\pi$ cannot be sorted by $t+1$ stacks in series. Indeed, we can never have all entries from a single $\beta^{\mathrm{r}}$ interval together in the first stack, because reading top to bottom they would form a copy of~$\beta$, and the remaining $t$ stacks could not sort them. Thus at least one entry from the first interval must exit the first stack before the first entry of the second interval enters, and similarly for subsequent intervals. Focusing on these $|\beta|-1$ entries together with the final entry of~$\pi$, we obtain a subsequence order-isomorphic to~$\beta$ that enters the last $t$ stacks in that order, and the last $t$ stacks cannot sort~$\beta$, so they cannot sort this subsequence.

Murphy's construction yields an unsortable permutation of length $|\beta|^2 - |\beta| + 1$. For three stacks, we start with a permutation of length $7$ (the shortest length not sortable by two stacks), giving an unsortable permutation of length $43$, which is longer than Tarjan's claimed example. Murphy~\cite[p.~329]{murphy:restricted-perm:} displays a permutation of length~$39$ obtained by deleting four entries from one such example, and suggests that it remains unsortable. In his 1992 technical report, Atkinson~\cite[Lemma~5]{atkinson:sorting-permuta:} had already done much the same, although he obtained an unsortable permutation of length~$38$. This record has stood for over thirty years.%
\footnote{See the final section for updates.}

\subsection*{Atkinson's $(r,s)$-stacks}

In a 1998 paper, Atkinson~\cite{atkinson:generalized-sta:} introduced a natural generalization of stacks. With a standard stack, one may push an entry onto the top or pop an entry from the top. An \emph{$(r,s)$-stack} relaxes these constraints: one may push the next input entry into any of the top $r$ positions and pop any of the top $s$ entries to the output. Thus a $(1,1)$-stack is an ordinary stack.

These machines are symbol oblivious, as their operations do not depend on the values of the entries, and the reverse of an $(r,s)$-stack is an $(s,r)$-stack. Thus, by Proposition~\ref{prop:sorting-dual}, we obtain the following.

\begin{proposition}[{Atkinson~\cite[Lemma~1.1]{atkinson:generalized-sta:}}]
\label{prop:rs-duality}
An $(r,s)$-stack sorts~$\pi$ if and only if an $(s,r)$-stack sorts its sorting dual~$(\pi^{\mathrm{rc}})^{-1}$.
\end{proposition}

In particular, the class of permutations sortable by an $(r,r)$-stack is closed under the sorting dual. We return to $(r,s)$-stacks later in this section.

\subsection*{Restricted containers ($\mathcal{C}$-machines)}

Atkinson's $(r,1)$-stacks are special cases of a broader family of sorting devices known as $\mathcal{C}$-machines. This framework captures many permutation classes and, as we will see, also includes his $(1,s)$-stacks via symmetry and duality.

Let $\C$ be a nonempty permutation class. A \emph{$\C$-machine} is a container that holds entries subject to the constraint that, at all times, the entries in the container (read left to right) must be order-isomorphic to a member of $\C$. In analyzing these machines, it turns out to be more natural to study generation than sorting; we discuss this choice further below. Thus we take the input to be $12 \cdots n$, and the machine has three operations:
\begin{itemize}[noitemsep, topsep=0pt]
\item \emph{push}: remove the next entry from the input and place it anywhere in the container, provided the resulting arrangement is order-isomorphic to a member of $\C$;
\item \emph{pop}: remove the leftmost entry from the container and append it to the output;
\item \emph{bypass}: remove the next entry from the input and append it directly to the output.
\end{itemize}
The permutation \emph{generated} by a sequence of operations is the output once all entries have exited. For any permutation class $\C$, the set of permutations that can be generated by a $\C$-machine also forms a permutation class.

Taking $\C = \Av(12)$ gives a machine that generates the same permutations as a classical stack. In this machine, entries in the container must be decreasing, so each push places the new entry on the left. Since pops also occur from the left, the bypass operation is redundant, and we recover the usual stack operations. As Knuth observed, this machine generates precisely $\Av(312)$.

By contrast, for the $\Av(21)$-machine, the entries in the container must be increasing, so each push places the new entry on the right. Since pops occur from the left, the bypass operation is now essential. This machine generates precisely $\Av(321)$: clearly it cannot generate $321$ or any permutation containing it, and conversely, to generate a $321$-avoiding permutation, output the left-to-right maxima with bypasses while passing the remaining entries (which must be increasing) through the container. (This is equivalent to generating with two queues in parallel.)

At the other extreme, taking $\C$ to be the class of all permutations allows entries to be placed anywhere in the container, and this machine generates all permutations.

The bypass operation is superfluous in some cases, as the stack example shows. However, including it greatly simplifies the basis theorem below. It also allows us to adopt a uniform convention for operating these machines. Every entry that bypasses the container is a left-to-right maximum of the generated permutation. Conversely, when generating a permutation that the machine can produce, we may always arrange to output the entries preceding a left-to-right maximum before reading that maximum from the input, and then bypass it. We may therefore assume that the entries that bypass the container are precisely the left-to-right maxima.

\subsection*{The basis theorem}

The main structural result for $\C$-machines characterizes the classes they generate and makes it immediately apparent whether a class of interest can be generated by a $\C$-machine, and if so, which one. The proof is too simple to omit. In the statement, we use the notation $1 \ominus B = \{1 \ominus \beta : \beta \in B\}$.

\begin{theorem}[{Albert, Homberger, Pantone, Shar, and Vatter~\cite[Theorem 1.1]{albert:generating-perm:}}]
\label{thm:c-machine-basis}
For any set~$B$ of nonempty permutations, the $\Av(B)$-machine generates the class $\Av(1 \ominus B)$.
\end{theorem}

\begin{proof}
The $\Av(B)$-machine cannot generate a permutation of the form $1 \ominus \beta$ for $\beta \in B$ because at the moment when the maximum entry (which appears first in $1 \ominus \beta$) is next in the input, the container would need to hold a copy of~$\beta$, and it is not allowed to do so. Because the $\Av(B)$-machine generates a permutation class, this shows that it can only generate permutations in $\Av(1 \ominus B)$.

For the converse, we show how to generate a permutation~$\pi$ that avoids all permutations in $1\ominus B$. Let $i_1 < i_2 < \cdots < i_\ell$ be the positions of the left-to-right maxima of~$\pi$; note that $i_1 = 1$ since the first entry is always a left-to-right maximum. Before we bypass each maximum, the container will hold precisely the entries southeast of it in~$\pi$, in their output order.

We proceed iteratively through the left-to-right maxima. For the first, we push all entries with values less than $\pi(i_1)$ into the container in their correct relative order with respect to~$\pi$. The resulting container avoids $B$, because its entries lie to the southeast of $\pi(i_1)$ in~$\pi$. Moreover, every intermediate container arising during these pushes is a pattern of this final container, and therefore also avoids $B$. Thus each of these pushes is legal. We then bypass $\pi(i_1)$ to the output.

For each subsequent left-to-right maximum, observe that all entries lying horizontally between the previous left-to-right maximum and the current one are already in the container. We pop these entries to the output. Then we push the entries with values between the two maxima from the input into the container, placing them in their correct relative positions with respect to~$\pi$. Again, this is possible because the entries lying to the southeast of the current left-to-right maximum must avoid~$B$. We then bypass the current maximum to the output, and repeat. After bypassing the final left-to-right maximum, we pop all entries remaining in the container.
\end{proof}

\subsection*{Sorting versus generating}

Unlike stacks, $\C$-machines need not be symbol oblivious: the available push operations can depend on the relative order of the entries in the container and the next entry to be input. Thus, we can't use Proposition~\ref{prop:sort-generate} to conclude that sorting and generating are symmetric. What would happen if we tried to sort with a $\C$-machine?

One answer is that sorting with $\C$-machines, as defined, would not be all that interesting. The only productive container configuration is increasing order, since any inversion in the container would persist to the output, resulting in a failure to sort. It follows that a $\C$-machine sorts the same permutations as a \emph{priority queue} whose capacity is determined by the increasing permutations contained in~$\C$. In a priority queue, one may push and pop entries, but instead of the last-in being popped (as in a stack) or the first-in being popped (as in a queue), the least-in, that is, the smallest entry, is popped. A priority queue of unlimited capacity can sort any permutation: simply push all the entries in, then pop them out in order.%
\footnote{The more interesting questions about priority queues concern the pairs $(\pi,\sigma)$ for which a priority queue can transform~$\pi$ into $\sigma$. A theorem of Atkinson and Thiyagarajah~\cite{atkinson:the-permutation:} states that there are precisely $(n+1)^{n-1}$ such pairs of permutations of length~$n$. By Cayley's formula, this is also the number of labeled trees on $n+1$ vertices. Atkinson and Thiyagarajah do not give a direct bijection; such bijections with labeled trees were subsequently given by Golin and Zaks~\cite{golin:labelled-trees-:}, Hamel~\cite{hamel:priority-queue-:}, and Atkinson and Beals~\cite{atkinson:priority-queues:}.}
Thus if $\C$ contains the identity of all lengths, the $\C$-machine can sort everything.

If $\C$ contains identity permutations only up to some length $k$, then sorting with the $\C$-machine is equivalent to sorting with a \emph{$(k+1)$-bounded priority queue}, which can hold at most $k+1$ entries at any time. (The ``$+1$'' accounts for the bypass operation.) Bounded priority queues have been studied by Atkinson and Tulley~\cite{atkinson:bounded-capacit:}, among others, though their sorting capabilities do not seem to have been written down explicitly. Each entry forces its larger predecessors to wait until it is output; greedily outputting the next required value therefore succeeds precisely when every entry has at most $k$ larger predecessors. Thus the sortable class is $\Av(S_{k+1} \ominus 1)$, consisting of the permutations avoiding all patterns of length~${k+2}$ that end with~$1$. Another way to describe this class is that it consists of the permutations with maximum drop size at most $k$, where the \emph{maximum drop size} of~$\pi$ is $\max\{i - \pi(i)\}$. Indeed, an entry maximizing the number of larger predecessors must be a right-to-left minimum, since any smaller entry to its right would have even more larger predecessors; at such an entry~$\pi(i)$, their number is $i-\pi(i)$. These classes are already very well understood; in fact, Chung, Claesson, Dukes, and Graham~\cite{chung:descent-polynom:} have even determined their descent polynomials. For example, the $\Av(12)$-machine generates the same class as an ordinary stack, but sorts only $\Av(231,321)$. In short, sorting with $\C$-machines as defined does not lead anywhere new.

An alternative approach to sorting with $\C$-machines would be to modify the definition of the machines to make sorting more natural. If a $\C$-machine can generate~$\pi$, then the reversed $\C$-machine (which pushes only at the left but allows pops from anywhere in the container) can transform $\pi^{\mathrm{r}}$ into $\id^{\mathrm{r}}$. By applying appropriate symmetries to $\C$, one could therefore study sorting by reversed $\C$-machines instead of generating by $\C$-machines. But this is clearly symmetric to generating with $\C$-machines, and generating is easier to talk about.

\subsection*{Atkinson's $(r,1)$- and $(1,s)$-stacks as $\C$-machines}

An $(r,1)$-stack allows pushing into any of the top $r$ positions but only popping from the top. By ``rotating'' this stack $90^\circ$ counterclockwise, we see that it generates the same permutations as the $\Av(S_r \oplus 1)$-machine. Indeed, the $\Av(S_r \oplus 1)$-machine allows insertion of the next entry from the input (which is larger than every entry currently in the container) into any of the first $r$ positions of the container, and popping from the first position of the container. By Theorem~\ref{thm:c-machine-basis}, this machine generates the class $\Av\bigl(1 \ominus (S_r \oplus 1)\bigr)$; taking inverses gives the class sortable by an $(r,1)$-stack, reproving a result of Atkinson~\cite[Theorem~2.1]{atkinson:generalized-sta:}. By Propositions~\ref{prop:sort-generate} and~\ref{prop:rs-duality}, we see that $(1,s)$-stacks also correspond to $\C$-machines, up to symmetry.

Atkinson~\cite[Section~3]{atkinson:generalized-sta:} went on to enumerate these classes, finding their algebraic generating functions and asymptotics. The enumeration of these classes also appears implicitly in the work of Kremer~\cite{kremer:permutations-wi:,kremer:postscript:-per:}, who considers, for each $r \ge 1$ and fixed indices $j, k \in [r+2]$ with $j\neq k$, classes with bases
\[
	B_{j,k} = \{\beta \in S_{r+2} : \text{$\beta(j) = r+1$ and $\beta(k) = r+2$}\}.
\]
Her main result~\cite[Theorem~1]{kremer:postscript:-per:} shows that all classes of the form $\Av(B_{j,k})$ with $|j - k| \le 2$, or $k = 1$, or $k = r+2$, have isomorphic generating trees, and thus are Wilf-equivalent. The underlying recurrences also appear in Sulanke's work~\cite{sulanke:three-recurrenc:} on colored parallelogram polyominoes. (We caution readers that the two generating functions displayed on p.~129 of Kremer's original paper~\cite{kremer:permutations-wi:} contain errors; Atkinson's paper should be consulted for the correct generating function.) For $r = 1$, we obtain the Catalan numbers. For $r = 2$, the enumeration gives the large Schr\"oder numbers (\OEISlink{A006318}), which arose earlier in this paper in the context of separable permutations; Kremer's classes are Wilf-equivalent to the separable permutations but not symmetric to them. For $r = 3$, we obtain \OEISlink{A054872}.

\subsection*{Sorting $1342$-avoiders}

As mentioned in Section~\ref{sec:wilf}, Atkinson, Murphy, and Ru\v{s}kuc~\cite{atkinson:sorting-with-tw:} showed that the class sortable by two increasing stacks in series is Wilf-equivalent to $\Av(1342)$. Here an \emph{increasing stack} is one whose contents must remain increasing when read from top to bottom.%
\footnote{Atkinson, Murphy, and Ru\v{s}kuc~\cite{atkinson:sorting-with-tw:} use the term ``ordered stack.'' We use ``increasing stack'' to distinguish these stacks from the decreasing stacks studied by Smith~\cite{smith:two-stacks-in-s:}.}
In sorting with stacks in series, the final stack must always be increasing, but the previous stacks do not need to be. At \emph{Permutation Patterns 2007}, B\'ona asked~\cite[Question~4]{vatter:problems-and-co:}:
\begin{quote}
Is there a natural sorting machine that can sort precisely the class $\Av(1342)$?
\end{quote}

Theorem~\ref{thm:c-machine-basis} shows that, up to symmetries of the class and swapping sorting for generating, the answer is yes: $\Av(4213)$ is generated by the $\Av(213)$-machine. However, what B\'ona actually had in mind was a more satisfying machine-level explanation of the Wilf-equivalence, and that remains open:

\begin{problem}
\label{prob:bona-bijection}
Find a bijection between $\Av(4213)$ and the class sortable by two increasing stacks in series that is witnessed by a correspondence between the operation sequences of the associated machines.
\end{problem}

\subsection*{Enumeration and D-finiteness}

The $\C$-machine framework leads to functional equations for generating functions. In favorable cases, these yield explicit formulas; in others, they allow efficient computation of many terms via dynamic programming.

A striking example is $\Av(4231, 4123, 4312)$. This class can be generated by a $\C$-machine (all basis elements begin with their maximum entries), and Albert, Homberger, Pantone, Shar, and Vatter~\cite{albert:generating-perm:} computed $5000$ terms of its counting sequence (\OEISlink{A257562}). Despite this abundance of data, no algebraic differential equation satisfied by the generating function has been found.

The counterexamples to the Noonan--Zeilberger conjecture~\cite{noonan:the-enumeration:} constructed by Garrabrant and Pak~\cite{garrabrant:permutation-pat:} have extremely large bases. A more compact counterexample would be nice to have, and $\Av(4231, 4123, 4312)$, with its basis of just three short permutations, appears to be a strong candidate.%
\footnote{See the final section for updates.}

\begin{conjecture}[Albert, Homberger, Pantone, Shar, and Vatter~{\cite[Conjecture~5.3]{albert:generating-perm:}}]
\label{conj:4231-4123-4312:not:diffalg}
The generating function for $\Av(4231, 4123, 4312)$ is not differentially algebraic.
\end{conjecture}

Amusingly, the superclass $\Av(4231, 4312)$ is enumerated by the large Schr\"oder numbers. This was conjectured by Stanley and first proved by Kremer~\cite[Proposition~11]{kremer:permutations-wi:}, although this class is not a member of her large family of Wilf-equivalent classes. Instead, she showed specifically that the generating tree for the symmetric class $\Av(2134, 1324)$ is isomorphic to the generating trees of her large family.

The class $\Av(4231)$, meanwhile, is symmetric to the notorious class $\Av(1324)$ mentioned in the introduction. Conway, Guttmann, and Zinn-Justin~\cite{conway:1324-avoiding-p:} have enumerated this class through length~$50$, which remains the record. Thus, in the chain
\[
	\Av(4231, 4123, 4312) \;\subseteq\; \Av(4231, 4312) \;\subseteq\; \Av(4231),
\]
the first class is computationally tractable but its generating function appears poorly behaved (and its enumeration is not a Stieltjes moment sequence, since its Hankel determinants eventually become negative~\cite[p.~40]{bostan:stieltjes-momen:}), the second has an algebraic generating function, and no one knows what to do with the third.

\subsection*{The $(2,2)$-stack}

Returning to Atkinson's $(r,s)$-stacks, the $(2,2)$-stack is the simplest case not captured by the $\C$-machine framework: it allows pushing into either of the top two positions and popping from either of the top two positions. Atkinson established the basis for this class.

\begin{theorem}[{Atkinson~\cite[Theorem~4.1]{atkinson:generalized-sta:}}]
\label{thm:22-stack-basis}
The class of $(2,2)$-stack-sortable permutations is
\[
	\Av(23451,\, 23541,\, 32451,\, 32541,\, 245163,\, 246153,\, 425163,\, 426153).
\]
\end{theorem}

By Proposition~\ref{prop:rs-duality}, this class is closed under the sorting dual. This can also be seen from the basis, as every element appears alongside its sorting dual:
\[
\begin{tabular}{r|cccccc}
\text{permutation} 
  & 23451 
  & 23541 
  & 32541 
  & 245163 
  & 246153 
  & 426153 \\\hline
\text{sorting dual}
  & \text{itself}
  & 32451
  & \text{itself}
  & \text{itself}
  & 425163
  & \text{itself}
\end{tabular}
\]
It would of course be desirable to have a general description of the basis of $(r,s)$-stack-sortable permutations, or even simply to know whether these bases are all finite.

\begin{problem}
\label{prob:rs-basis}
Characterize the basis of the class sortable by an $(r,s)$-stack for general~${r, s \ge 2}$.%
\footnote{See the final section for updates.}
\end{problem}

The permutations sortable by a $(2,2)$-stack are counted by \OEISlink{A393395}. The counts for positive lengths begin
\[
	1, 2, 6, 24, 116, 628, 3636, 21956, 136428, 865700, 5583580, 36490740, \dots.
\]
Atkinson reported a conjectured expression from the OEIS Superseeker for the generating function of these permutations but did not prove it. Using Combinatorial Exploration~\cite{albert:combinatorial-e:} and working from the basis in Theorem~\ref{thm:22-stack-basis}, Pantone (personal communication) proved Atkinson's conjecture, establishing that the generating function $f(x)$ for this class, with $f(0)=1$, has minimal polynomial
\[
	2xf^3 - (2x+3)f^2 - (x-7)f - 4.
\]
(This differs slightly from Atkinson's presentation because he omitted the constant term~$1$ for the empty permutation, but it is included in $f(x)$.) It would still be of interest to derive this enumeration directly from the sorting mechanism.

\begin{problem}
Derive the generating function for the class of $(2,2)$-stack-sortable permutations from the structure of the $(2,2)$-stack.%
\footnote{See the final section for updates.}
\end{problem}

Given that all known enumerations of $(r,s)$-stack-sortable classes are algebraic, one might ask how far this extends.

\begin{question}
Does the class of permutations sortable by an $(r,s)$-stack have an algebraic generating function for all $r, s \ge 1$?\,%
\footnote{See the final section for updates.}
\end{question}

\subsection*{Generalizing $\C$-machines}

The $(2,2)$-stack is not a $\C$-machine: it allows popping from either of the top two positions, whereas $\C$-machines pop only from the leftmost position. This suggests a natural generalization of $\C$-machines: allow the pop operation to remove any of the leftmost $s$ entries from the container, rather than just the leftmost one. Call such a machine an \emph{extended $\C$-machine with $s$-pop}, or a \emph{$(\C, s)$-machine} for short. The $\C$-machines considered earlier are then $(\C, 1)$-machines, and the $(2,2)$-stack generates the same permutations as the $(\Av(123, 213), 2)$-machine; by Proposition~\ref{prop:sort-generate}, these are the inverses of the $(2,2)$-stack-sortable permutations described in Theorem~\ref{thm:22-stack-basis}.

As with $\C$-machines, the permutations generated by a $(\C, s)$-machine form a class: to generate a subpermutation, run the same operations while ignoring the missing entries. The basis theorem (Theorem~\ref{thm:c-machine-basis}) gives a complete characterization of the classes generated by $(\C, 1)$-machines. Does a similar result hold more generally?

\begin{problem}
\label{prob:pop-k-basis}
Is there a basis theorem for $(\C, s)$-machines analogous to Theorem~\ref{thm:c-machine-basis}?
\end{problem}

Even partial progress would be valuable. For instance, if $\C$ is finitely based, does every $(\C, s)$-machine generate a class with a finite basis?%
\footnote{See the final section for updates.}

\section*{Updates and acknowledgements}

I am grateful to Jay Pantone for numerous extremely helpful discussions about the material presented here, to Bruce Sagan for catching an error in an earlier draft, to Colin Defant and Yiran Zhao for pointing out the consequences of Hoffman, Rizzolo, and Slivken's work for derangement proportions, and to Natasha Blitvi\'{c} and Einar Steingr\'{\i}msson for helpful correspondence about their continued fraction. I also thank the anonymous referees for comments and suggestions that improved the paper. Various LLMs, including ChatGPT and Claude, were invaluable for fact-checking and copy-editing.

The following notes record subsequent work on these problems of which I am aware at the time of this revision. This list makes no claim to completeness. Some of this work is still in progress, and not all of it is yet available in preprint form.

\begin{itemize}

\item
Paul Zinn-Justin has announced proofs that the generating functions for $\Av(4123,4312)$, one of the three exceptional $2\times4$ classes discussed in Section~\ref{sec:derangements}, and $\Av(4231,4123,4312)$ are not D-finite. These provide simpler counterexamples to the Noonan--Zeilberger conjecture than those constructed by Garrabrant and Pak~\cite{garrabrant:permutation-pat:}. Albert, Homberger, Pantone, Shar, and Vatter~\cite[Conjectures~5.3 and~5.4]{albert:generating-perm:} conjectured that the generating functions for both classes are not differentially algebraic; Zinn-Justin's announced results do not settle these stronger assertions. The same authors also conjectured that the generating functions for $\Av(4123,4231)$ and $\Av(4231,4321)$, the other two exceptional $2\times4$ classes, are not differentially algebraic.

\item
Lizhong Chen has announced partial results on the rank-unimodality of composition downsets, the subject of Conjecture~\ref{conj-layered-unimodal}.

\item
Alexander Dobner has reported an answer to Question~\ref{ques:we:len8}, finding that the permutations of length~$8$ form $4755$ Wilf-equivalence classes. Eric Cowan has reported an independent confirmation of this count.

\item
Eric Cowan has reported complete classifications of the shape-Wilf-equivalence classes of single patterns of lengths~$7$ and~$8$, finding $4862$ and $39326$ classes, respectively. In both cases, all equivalences are implied by Proposition~\ref{prop:swe-sums} and Theorems~\ref{thm:inc-dec-swe} and~\ref{thm:231-312-swe}.

\item
There has been further work on the Ray--West statistic of Problem~\ref{prob:ray:west}, including by Eric Cowan, Souparna Majumder, and Donald Poindexter. Cowan has also announced work on a formula for $g_3(\beta)$.

\item
Qixuan Fu has reported a proof that $(-1)^{n/2}G_n(-1)>0$ for every even~$n\ge2$, establishing the sign assertion in Fu, Tang, Han, and Zeng's conjecture~\cite[Conjecture~5.11]{fu:qt-catalan-numb:}.

\item
Alex Chengyu Li has announced work on Birmajer, Gil, Tirrell, and Weiner's conjecture~\cite[Conjecture~A.2]{birmajer:pattern-avoidin:} concerning the distance between the two fixed points of a $123$-avoiding permutation.

\item
Yiran Zhao has reported a more compact formula for the two-fixed-point count in Elizalde's work~\cite[Theorem~3.3]{elizalde:multiple-patter:}. His approach also gives a direct enumerative proof of the limiting proportion~$\nicefrac{9}{16}$ of derangements among $123$-avoiding permutations (Theorem~\ref{thm:123-derangement-proportion}).

\item
Wenlin Zhang and Haobo Ma has shown that there is no largest attainable limiting proportion of derangements strictly below~$1$. Let $\mathcal{I}=\Av(21)$ denote the class of increasing permutations, and for each integer $k\ge0$, let $\mathcal{D}_{\le k}=\Av(12, (k+1)\cdots 21)$ denote the class of decreasing permutations of length at most~$k$. Then $\mathcal{I}\ominus\mathcal{D}_{\le k}$ is a permutation class, and for $n>2k$, its $k+1$ members of length~$n$ consist of the identity and~$k$ derangements, giving limiting proportion $k/(k+1)$. These limits approach~$1$ from below.

\item
Robert Dougherty-Bliss, Alejandro B.~Galv\'an, Jay Pantone, Michaela Polley, and David Shuster have investigated the enumeration and asymptotics of separable derangements. Their computations through length~$3000$ lead them to conjecture that the limiting proportion of derangements among separable permutations (Question~\ref{ques:sep:derange:ratio}) is~$0$, rather than between~$\nicefrac{1}{5}$ and~$\nicefrac{1}{4}$ as speculated here. Eric Cowan has also announced work on fixed points and derangements in separable permutations.

\item
Jay Pantone and I have announced progress on Question~\ref{ques:shortest:3ss}, narrowing the possible length of the shortest permutation that cannot be sorted by three stacks in series to between~$17$ and~$21$: every permutation of length~$16$ is sortable, and we have found an unsortable permutation of length~$21$. In particular, both Elder's and Waton's guesses are incorrect. The length-$21$ example also disproves the $t=3$ case of Murphy's conjecture that $t$ stacks in series can sort all permutations of length up to $(t+1)!$. However, as discussed after Question~\ref{ques:shortest:3ss}, that conjecture was already known to fail for sufficiently large~$t$.

\item
Columbus Labs has announced a proof that the generating function for $(r,s)$-stack-sortable permutations is algebraic for every fixed pair of positive integers~$r,s$, answering Question~5.12. Companion papers report a direct derivation of the generating function for $(2,2)$-stack-sortable permutations and a proof that the bases considered in Problem~\ref{prob:rs-basis} are infinite whenever $r,s\ge2$ and $(r,s)\ne(2,2)$. The latter paper also reports a negative answer to the finite-basis question for extended $\C$-machines.

\end{itemize}


\def\cprime{$'$}

\end{document}